\documentclass[11pt,oneside,reqno]{article}

\usepackage{amsmath,amsthm,amssymb,color,bm,graphicx, float}

\usepackage{array}
\usepackage[margin=2.5cm,a4paper]{geometry}
\usepackage{bbm}
\usepackage{paralist}
\usepackage{mathrsfs}
\usepackage[breaklinks=true]{hyperref}

\usepackage[square]{natbib} %[authoryear,longnamesfirst]{natbib}

\renewcommand{\cite}{\citep*}

\numberwithin{equation}{section}
\allowdisplaybreaks[4]

\theoremstyle{plain}
\newtheorem{theorem}{Theorem}[section]

\newtheorem{lemma}[theorem]{Lemma}

\newtheorem{conjecture}[theorem]{Conjecture}

\theoremstyle{definition}

\newtheorem{remark}[theorem]{Remark}

\makeatletter

\renewcommand{\phi}{\varphi}
\newcommand{\eps}{\epsilon}
\newcommand{\eq}{\eqref}

\newcommand{\bigo}{\mathrm{O}}
\newcommand{\lito}{\mathrm{o}}

\newcommand{\Beta}{\mathop{\mathrm{Beta}}}

\newcommand{\Exp}{\mathop{\mathrm{Exp}}}
\newcommand{\TI}{\mathop{\mathrm{TI}}}
\def\E{\mathbbm{E}}
\newcommand{\Var}{\mathop{\mathrm{Var}}\nolimits}
\newcommand{\Cov}{\mathop{\mathrm{Cov}}}

\newcommand{\e}{{\mathrm{e}}}

\newcounter{ctr}\loop\stepcounter{ctr}\edef\X{\@Alph\c@ctr}%
	\expandafter\edef\csname s\X\endcsname{\noexpand\mathscr{\X}}
	\expandafter\edef\csname c\X\endcsname{\noexpand\mathcal{\X}}
	\expandafter\edef\csname b\X\endcsname{\noexpand\boldsymbol{\X}}
	\expandafter\edef\csname I\X\endcsname{\noexpand\mathbbm{\X}}
	\expandafter\edef\csname r\X\endcsname{\noexpand\mathrm{\X}}
\ifnum\thectr<26\repeat

\def\ba#1{\begin{align*}#1\end{align*}}
\def\ban#1{\begin{align}#1\end{align}}

\def\given{\typeout{Command 'given' should only be used within bracket command}}
\newcounter{@bracketlevel}
\def\@bracketfactory#1#2#3#4#5#6{
\expandafter\def\csname#1\endcsname##1{%
\addtocounter{@bracketlevel}{1}%
\global\expandafter\let\csname @middummy\alph{@bracketlevel}\endcsname\given%
\global\def\given{\mskip#5\csname#4\endcsname\vert\mskip#6}\csname#4l\endcsname#2##1\csname#4r\endcsname#3%
\global\expandafter\let\expandafter\given\csname @middummy\alph{@bracketlevel}\endcsname
\addtocounter{@bracketlevel}{-1}}%
}
\def\bracketfactory#1#2#3{%
\@bracketfactory{#1}{#2}{#3}{relax}{1mu plus 0.25mu minus 0.25mu}{0.6mu plus 0.15mu minus 0.15mu}
\@bracketfactory{b#1}{#2}{#3}{big}{1mu plus 0.25mu minus 0.25mu}{0.6mu plus 0.15mu minus 0.15mu}
\@bracketfactory{bb#1}{#2}{#3}{Big}{2.4mu plus 0.8mu minus 0.8mu}{1.8mu plus 0.6mu minus 0.6mu}
\@bracketfactory{bbb#1}{#2}{#3}{bigg}{3.2mu plus 1mu minus 1mu}{2.4mu plus 0.75mu minus 0.75mu}
\@bracketfactory{bbbb#1}{#2}{#3}{Bigg}{4mu plus 1mu minus 1mu}{3mu plus 0.75mu minus 0.75mu}
}
\bracketfactory{clc}{\lbrace}{\rbrace}
\bracketfactory{clr}{(}{)}
\bracketfactory{cls}{[}{]}

\def\abs#1{\vert#1\vert}
\def\babs#1{\bigl\vert#1\bigr\vert}
\def\bbabs#1{\Bigl\vert#1\Bigr\vert}
\def\bbbabs#1{\biggl\vert#1\biggr\vert}

\renewcommand\section{\@startsection {section}{1}{\z@}%
{-3.5ex \@plus -1ex \@minus -.2ex}%
{1.3ex \@plus.2ex}%
{\center\small\sc\mathversion{bold}\MakeUppercase}}

\def\subsection#1{\@startsection {subsection}{2}{0pt}%
{-3.5ex \@plus -1ex \@minus -.2ex}%
{1ex \@plus.2ex}%
{\bf\mathversion{bold}}{#1}}

\def\subsubsection#1{\@startsection{subsubsection}{3}{0pt}%
{\medskipamount}%
{-10pt}%
{\normalsize\itshape}{\kern-2.2ex. #1.}}

\def\blfootnote{\xdef\@thefnmark{}\@footnotetext}

\makeatother

\newcommand\ed{\stackrel{d}{=}}

\def\l{\lambda}

\def\BC{\mathrm{BC}}

\def\e{\eps}

\def\Bin{\text{Bin}}
\def\Hg{\text{Hg}}
\def\a{\alpha}

\begin{document}

\title{\sc\bf\large\MakeUppercase{Stationary distribution approximations of Two-island Wright-Fisher and seed-bank models using Stein's method}}
\author{\sc Han~L.~Gan and Maite Wilke--Berenguer}
\date{\it University of Waikato and Humboldt-Universit\"at zu Berlin}
\maketitle

\begin{abstract}
We consider two finite population Markov chain models, the two-island Wright-Fisher model with mutation, and the seed-bank model with mutation. Despite the relatively simple descriptions of the two processes, the the exact form of their stationary distributions is in general intractable. For each of the two models we provide two approximation theorems with explicit upper bounds on the distance between the stationary distributions of the finite population Markov chains, and either the stationary distribution of a two-island diffusion model, or the beta distribution. We show that the order of the bounds, and correspondingly the appropriate choice of approximation, depends upon the relative sizes of mutation and migration. In the case where migration and mutation are of the same order, the suitable approximation is the two-island diffusion model, and if migration dominates mutation, then the weighted average of both islands is well approximated by a beta random variable. Our results are derived from a new development of Stein's method for the stationary distribution of the two-island diffusion model for the weak migration results, and utilising the existing framework for Stein's method for the Dirichlet distribution. 
\end{abstract}

% \noindent\textbf{Keywords: } 

\section{Introduction}
The Wright-Fisher model, originating in~\cite{Fisher, Wright31} is one of the canonical probabilistic models used in mathematical population genetics. In this work, we study the two-island model, which allows for the population to be subdivided into two islands, with migration between them. Our goal is to study the approximation of stationary distributions for finite population two-island Wright-Fisher Markov chains with the stationary distributions of scaling diffusion limits. We begin by carefully defining the two models. 

The finite population two-island Wright-Fisher Markov chain generalises the Wright-Fisher model by having two islands with their individual Wright-Fisher dynamics, and allowing migration between the two. The model is a Markov chain used to study genetic type frequencies from generation to generation for two interacting populations of fixed sizes. The formal mathematical definition of the model is as follows. 

Given $(X,Y)$ let $(X',Y')$ be one step ahead in Markov chain. For $N, M, c \in \mathbb{Z}^+$, $p_1, p_2, q_1, q_2 \in (0,1)$, and $1 \leq c \leq \min\{M,N\}$, define the following conditionally independent random variables,
\ba{
A|(X,Y) &\sim \Bin(N-c, p_1 + (1-p_1 - p_2)X),\\
B|(X,Y) &\sim \Bin(c,p_1 + (1-p_1 - p_2)X),\\
C|(X,Y) &\sim \Bin(c, q_1 + (1-q_1-q_2)Y),\\
D|(X,Y) &\sim \Bin(M-c, q_1 + (1-q_1-q_2)Y).
}
Then $(NX',MY') | (X,Y) = (A + C, B + D) | (X,Y)$.

We briefly describe a standard interpretation of the mathematical formulation. Given two islands of fixed size $N$ and $M$ individuals, we consider a haploid population where individuals take genetic type $1$ or $2$, and set $(X_t, Y_t)$ as the Markov chain that tracks the proportion of type $1$ on the two respective islands from one generation to the next. Given the parent generation, we think of the process that generates the next offspring generation in a two step manner, first the offspring choose a parent, and then their type is decided by the type of their parents and a potential mutation step.

For the first island, independently and uniformly at random, $N-c$ offspring chose a parent from the first island, and $c$ offspring chose a parent from the second island. Similarly in the second island, $M-c$ parents are chosen from the second island, and $c$ chosen are from the first island. The parameter $c$ therefore represents the amount of migration between the two islands from one generation to the next. The types of the offspring depend upon the following rules. For offspring of parents from first island, if the chosen parent is of type $1$ $(2)$, then the child mutates to type $2$ $(1)$ with probability $p_2$ $(p_1)$ or otherwise has the same type as its parent. For offspring of parents from the second island, we follow the same rules except with mutation probabilities $q_1$ and $q_2$.

The seed-bank model is similar in motivation to the two-island Wright-Fisher model in that it is used to model genetic type frequencies from generation to generation for two interacting populations of fixed size, but it differs from the two-island Wright-Fisher model, by disabling reproduction and mutation in the second island. 

Given $(X,Y)$ let $(X',Y')$ be one step ahead in Markov chain. For $N, M, c \in \mathbb{Z}^+$, $p_1, p_2 \in (0,1)$, and $1 \leq c \leq \min\{M,N\}$,
\ba{
A|(X,Y) &\sim \Bin(N-c, p_1 + (1-p_1 - p_2)X),\\
B|(X,Y) &\sim \Bin(c,p_1 + (1-p_1 - p_2)X),\\
C|(X,Y) &\sim \Hg(M,MY,c),\\
D|(X,Y) &\sim \Hg(M,MY,M-c).
}
Importantly, the above random variables are pairwise conditionally independent except $D|(X,Y) = M-C|(X,Y)$.

We again briefly describe a standard interpretation of the mathematical formulation. Let the populations be of sizes $N$ and $M$, and we will call them the active population and seed-bank population respectively. We again consider a haploid population with types $1$ or $2$, and set $(X_t, Y_t)$ as the Markov chain that tracks the proportion of type $1$ on the two respective populations from one generation to the next. Given the parent generation, the dynamics are described as follows.

As per the two-island model, let $c$ represent the migration rate. For the active population, $N-c$ offspring choose a parent from the first island independently and uniformly at random with replacement. $c$ individuals from the seed-bank parent generation are chosen uniformly \emph{without} replacement to migrate to the active population. For the seed-bank, the remaining $M-c$ individuals who did not migrate remain `dormant' and stay over into the next generation, and $c$ offspring are sampled uniformly at random with replacement from the active population. Mutation occurs, but only for offspring of the active population. As in the two-island model, if the chosen parent is of type $1$ $(2)$, then the child mutates to type $2$ $(1)$ with probability $p_2$ $(p_1)$ or otherwise has the same type as its parent. Individuals that migrate from the seed bank do not change their types.

\emph{Seed banks} arise in populations who exhibit \emph{dormancy},  which describes de capability of an individual to enter and exist a state of reduced metabolic activity during which it cannot reproduce,  but is at the same time better protected from death, e.g. through difficult environmental conditions. The set of dormant individuals is then referred to as a seed bank. This mechanism is particularly prevalent among microbial populations, but has independently developed several times across the tree of life (cf. \cite{LJ11,SL18}). It has in recent years received increased attention and become an active field of research in population genetics and beyond, see  \cite{NatCom}  and references therein. One key parameter is the time-scale of dormancy compared to that of reproduction (or the population size).  The model for dormancy presented above (and introduced in \cite{Blathetal2016}) describes a so-called \emph{strong} seed bank effect,  where dormancy occurs at the same scale as reproduction. In this regime, the frequency process described above converges to the so-called seed-bank diffusion, while the ancestry converges to the seed bank coalescent \cite{Blathetal2016}.

Closed form solutions for the stationary distributions for both of the two-island and seed-bank models is known to be a difficult problem with no nice solutions \cite{Blathetal2019}, and hence good approximations are of particular value. In this paper we provide two approximations each for both models, namely the stationary distribution of the two-island Wright-Fisher diffusion model, and the beta distribution. All of the bounds for the approximations are explicit, including the constants, and are general for any possible values of the parameters $N, M, p_1, p_2, q_1, q_2, c$. 

In this paper we investigate approximations for the stationary distributions of these processes, in particular as the population sizes grow. We restrict the scope of this paper to two allele type models with mutation and only two islands, but note that generalising to multiple types and multiple islands would not require any additional theoretical mathematical complexity, but significant additional tedious bookkeeping notation and calculations. The two main driving forces in the models are mutation and migration. As a result there are three main scenarios to consider. Namely if mutation dominates migration, if migration dominates mutation, or if mutation and migration are of the same order. In this paper, we focus on the latter two scenarios, as the first case is relatively uninteresting as the islands effectively operate independently in this scenario. More precisely, to assess the accuracy and applicability of the approximations we consider the following asymptotic regimes for the mutation and migration rates. We will follow common population genetics terminology and call mutation \emph{weak} if mutation is rare in the sense that the mutation probabilities $p_1, p_2$ are of order $\bigo(1/N)$ and similarly for $q_1, q_2$. Throughout the rest of this paper we will be assuming that mutation is weak. In contrast, we will vary the level of migration. We will consider migration to be of the order $\bigo(N^\eps)$ where $\eps \in [0,1]$. If $\eps = 0$, this corresponds to a constant order of migration, which we will also therefore call weak. If $\eps > 0$, then we will call migration strong.

Assuming weak mutation, the appropriate scaling diffusion limit depends upon whether migration is strong or weak. We show that under strong migration, the two islands are `well-mixed', and the appropriate approximation is to approximate the stationary distribution of the weighted average of the two islands with a beta distribution. In contrast, under weak migration, each island maintains its individual characteristics and the appropriate approximating distribution is the stationary distribution of the two-island diffusion model. One approach to understand this dichotomy is to consider the coalescent process under weak or strong migration, see for example~\cite{NK2002}. If migration occurs on the same scaling as mutation, then lineages will usually coalesce without migration. On the other hand with strong migration, migration occurs a larger number of times before coalescence, and in the limit the standard one-island coalescent is achieved.

Before we present the main results, we first define the two distributions we use as reference measures. We then describe the general order the bounds for the approximations for both models and both distributions.
\subsection{Approximating distributions}
Denote by $\text{Beta}(a_1, a_2)$ the Beta distribution with parameters $a_1>0$ and $a_2>0$ defined on the unit interval $[0,1]$ and density function
\ba{
\psi(x) = \frac{\Gamma(a_1 + a_2)}{\Gamma(a_1) \Gamma(a_2)}x^{a_1 -1}(1-x)^{a_2-1}.
}

For our beta approximation results, we first define some technical conditions. For a full discussion of these definitions see~\cite{GRR17}. For $m \ge 1$, let $\BC^{m,1}([0,1])$ be the class of bounded functions on $[0,1]$ that have $m$ bounded and continuous derivatives  
and whose $m^{\mathrm{th}}$ derivative is Lipschitz continuous.
For some $h$ in $\BC^{m,1}([0,1])$, let $\Vert h\Vert_{\infty}$ be the sup-norm of $h$, and if the $k^{\mathrm{th}}$ derivative of $h$ exists, let 
\begin{equation*}
|h|_k := \left\Vert \frac{d^k h}{d x^k} \right\Vert_{\infty} \qquad \text{and} \qquad |h|_{k,1} :=  \sup_{x,y} \left| \frac{d^k h(x)}{d x^k} - \frac{d^k h(y)}{d y^k}  \right| \frac{1}{|x - y|}.
\end{equation*}
Our beta approximation error bounds will be of the following form. Given a random variable $X \in [0,1]$, and $Z \sim \Beta(a_1, a_2)$,  for any $h \in \BC^{2,1}[0,1]$, we will derive bounds for $|\E h(X) - \E h(Z)|$.

To define the stationary distribution for the two-island diffusion model, we first define the generator of the diffusion. For functions $f \in \mathrm{C}^2[0,1]^2$, the generator of the two-island diffusion is as follows,
\ban{
\cA f(x,y) &= [ a_1 - (a_1 + a_2)x + c_1(y-x)] f_x(x,y) + \frac\alpha2x(1-x) f_{xx}(x,y) \notag\\
	&\ \ \ +  [ b_1 - (b_1 + b_2)y + c_2(x-y)] f_y(x,y) + \frac\beta2y(1-y) f_{yy}(x,y).\label{eq:gen}
}
Let $\TI:=\TI(a_1, a_2, b_1, b_2, c_1, c_2, \alpha, \beta)$ denote the stationary measure associated with two-island diffusion, and we will often suppress the parameters for brevity. 

Despite the relatively simple formulation of the generator of the two island diffusion~\eq{eq:gen}, in general the exact form of the stationary distribution TI remains unknown, see~\cite{BurdenGriffiths2019} for example and for some partial solutions. However, the moments of $\TI$ can be easily computed by taking expectations of~\eq{eq:gen} with respect to polynomial test functions, see~\cite[Lemma~2.1]{Blathetal2019} for example. In this paper we will, for any random vector $(X,Y)\in [0,1]^2$ and any polynomial function $h: [0,1]^2 \mapsto \IR$, aim to calculate explicit bounds for $\abs{\IE h(X, Y) - \IE h(Z_1, Z_2)},\label{eq:aim}$
where $(Z_1,Z_2) \sim \TI$. Given only the moments of $\TI$ are explicitly known, it suffices to focus on test functions $h$ that take the form $h(x,y) = x^{n}y^{m}$, where $n, m$ are non-negative integers. 

\subsection{Main results}
\begin{theorem}\label{thm:WFs}
Let $(X,Y)$ be distributed as the proportions of type $1$ in both islands for the stationary distribution of the finite population two-island Wright-Fisher model. If mutation is weak and migration is $\bigo(N^\eps)$ for $\eps \in [0,1]$, then:

\begin{enumerate}
\item Set $(Z_1,Z_2) \sim \TI$ with parameters $a_1 = 2p_1(N-c)$, $a_2 = 2p_2(N-c)$, $b_1 = 2\frac NM (M-c){M}q_1$, $b_2 = 2\frac NM (M-c) q_2$, $c_1 = 2c$, $c_2 = \frac{2cN}{M}$, $\alpha = 2$ and $\beta = \frac{2N}{M}$. Then for any $h(x,y) = x^{n}y^{m}$,
\ba{
\abs{\E h(X,Y) - \E h(Z_1, Z_2)} &\leq \bigo(N^{\max\{2\eps -1, -1/2\}}).
}
\item Set $Z \sim \Beta(a_1,a_2)$ with parameters $a_1 = 2(Np_1 + Mq_1)$ and $a_2 = 2(Np_2 + Mq_2)$. Then for any $h \in \BC^{2,1}[0,1]$,
\ba{
\left| \E h\left(\frac{N}{N+M}X +\frac{M}{N+M} Y\right) - \E h(Z) \right| \leq  \bigo(N^{-\eps/2}).
}
\end{enumerate}
\end{theorem}

\begin{theorem}\label{thm:SBs}
Let $(X,Y)$ be distributed as the proportions of type 1 for the stationary distribution of the finite population seed-bank model. If mutation is weak and migration is $\bigo(N^\eps)$ for $\eps \in [0,1]$, then:
\begin{enumerate}
\item Set $(Z_1,Z_2) \sim \TI$ with  with parameters $a_1 = 2p_1(N-c)$, $a_2 = 2p_2(N-c)$, $b_1 = b_2=0$, $c_1 = 2c$, $c_2 = \frac{2cN}{M}$, $\alpha = 2$ and $\beta = 0$. Then for $h(x,y) = x^{n}y^{m}$,
\ba{
\abs{\E h(X,Y) - \E h(Z_1, Z_2)} &\leq \bigo(N^{\max\{2\eps -1, -1/2\}}).
}
\item Set $Z \sim \Beta(a_1,a_2)$ with parameters $a_1 = 2(N+ M)p_1$ and $a_2 = 2(N + M)p_2$. Then for any $h \in \BC^{2,1}[0,1]$,
\ba{
\left| \E h\left(\frac{N}{N+M}X +\frac{M}{N+M} Y\right) - \E h(Z) \right| \leq  \bigo(N^{-\eps/2}).
}
\end{enumerate}
\end{theorem}
From these two main theorems, we can now easily see that when migration is weak ($\eps = 0$), then a two-island approximation is appropriate, and if migration is strong ($\eps > 0$), then the Dirichlet approximation is more suitable.

 For the remainder of this paper, Section~2 will consist of the main results in full, consisting of the explicit bounds of the approximations. It will also include an additional discussion of the results. The final section will service as an appendix that includes all of the technical proofs of the results in Section~2.
\section{Full results \& additional discussion}
The main results are derived using Stein's method, in particular the exchangeable pairs approach. The final bounds, while initially may appear to be intimidatingly long, have a relatively straightforward interpretation. One interpretation for Theorems~\ref{thm:WF} and~\ref{thm:SB}, is via the generator approach that originated in~\cite{Barbour1990}. If the finite population two-island Wright-Fisher Markov chain and the two-island diffusion model have similar stationary distributions, then their generators should also be similar. Each term of the final bounds quantify the differences between the continuous generator and a Taylor expansion of the discrete generator. The first order terms $A_x, A_y$ measure the differences in the ``drift'' components, the second order terms $A_{xx}, A_{yy}, A_{xy}$ measure the differences in the ``diffusion'' components, the remaining terms bound higher order error terms. An analogous interpretation applies for Theorems~\ref{thm:Beta} and~\ref{thm:SBB} and their beta approximations which are simpler due to the projection to the one dimensional approximating distribution. 
\begin{theorem}\label{thm:WF}
Let $(X,Y)$ be distributed as the proportions of type $1$ in both islands for the stationary distribution of the finite population two-island Wright-Fisher model, and $(Z_1,Z_2) \sim \TI$ with parameters $a_1 = 2p_1(N-c)$, $a_2 = 2p_2(N-c)$, $b_1 = 2\frac NM (M-c){M}q_1$, $b_2 = 2\frac NM (M-c) q_2$, $c_1 = 2c$, $c_2 = \frac{2cN}{M}$, $\alpha = 2$ and $\beta = \frac{2N}{M}$. Then for $h(x,y) = x^{n}y^{m}$,
\ba{
\abs{\E h(X,Y) - \E h(Z_1, Z_2)} &\leq D_x A_x + D_y A_y + D_{xx}A_{xx} + D_{yy}A_{yy} + D_{xy}A_{xy} \\
	&\hspace{0.5cm}+ D_{xxx}A_{xxx} + D_{xxy}A_{xxy} + D_{xyy}A_{xyy} + D_{yyy}A_{yyy}
}
where
\ba{
A_x &= 2c(q_1 + q_2),\\
A_y &=\frac{2cN}{M}(p_1 + p_2),\\
A_{xx} &= N(p_1+p_2)^2 + (4c + 1 + 2Np_1 + 2cq_1)(p_1+p_2) + p_1 + Np_1^2 + 2cp_1(q_1+q_2)\\
	&\hspace{0.5cm} + \frac{1}{N}(2c^2(1+p_1+q_1)),\\
A_{yy} &=\frac{N}{M} \Big\{M(q_1+q_2)^2 + (4c + 1 + 2Mq_1 + 2cp_1)(q_1+q_2) + q_1 + Mq_1^2 + 2cq_1(p_1+p_2)\\
	&\hspace{0.5cm} + \frac{1}{M}(2c^2(1+p_1+q_1)). \Big\},\\
A_{xy} &=2N\Big( p_1 + p_2 + \frac{c}{N} (1 + q_1 + q_2) \Big) \Big( q_1 + q_2 + \frac{c}{M} (1 + p_1 + p_2) \Big),\\
A_{xxx} &=  \frac16\Bigg[  \frac{1}{N} \big(  4N^2 + 4c^2 + 6cN \Big)^\frac{1}{4} +\eps_x \Bigg]^2 \Bigg[ 2\sqrt{N} + 2N \eps_X \Bigg],\\
A_{xxy}&=\frac12\Bigg[  \frac{1}{N} \big(  4N^2 + 4c^2 + 6cN \big)^\frac{1}{4} +\eps_x \Bigg]^2 \Bigg[ \frac{2N}{\sqrt M} + 2N\eps_Y \Bigg],\\
A_{xyy}&= \frac12 \Bigg[  \frac{1}{M} \big(  4M^2 + 4c^2 + 6cM \big)^\frac{1}{4} +\eps_y \Bigg]^2 \Bigg[ 2\sqrt N +2 N\eps_X \Bigg],\\
A_{yyy}&=\frac16\Bigg[  \frac{1}{M} \big(  4M^2 + 4c^2 + 6cM \big)^\frac{1}{4} +\eps_y \Bigg]^2 \Bigg[ \frac{2N}{\sqrt M} + 2N\eps_Y \Bigg],\\
\eps_x &=  (p_1 + p_2) + \frac{c}{N}(1 + q_1 + q_2),\\
\eps_y &=  (q_1 + q_2) + \frac{c}{M}(1 + p_1 + p_2),
}
and $D_x, D_y$, etc.~are defined in Theorem~\ref{thm:fcts}.
\end{theorem}

\begin{theorem}\label{thm:Beta}
Let $(X,Y)$ be distributed as the proportions of type $1$ in both islands for the stationary distribution of the finite population two-island Wright-Fisher model, and $Z \sim \text{Beta}(a_1, a_2)$ with parameters $a_1 = 2(Np_1 + Mq_1)$ and $a_2 = 2(Np_2 + Mq_2)$. Then for any $h \in \BC^{2,1}([0,1])$,
\ba{
\Bigg| \E h\Bigg( \frac{N}{N+M} X + \frac{M}{N+M}Y\Bigg) - \E h(Z) \Bigg| \leq \frac{|h|_1}{a_1 + a_2} A_1 + \frac{|h|_2}{2(a_1 + a_2 + 1)}A_2 + \frac{|h|_{2,1}}{18(a_1 + a_2 + 2)} A_3,
}
where
\ba{
A_1 &=\frac{2NM}{(N+M)^2} | (q_1 + q_2) - (p_1 + p_2) | \sqrt{ \E(X - Y)^2}, \\
A_2 &= \frac{1}{N+M} \Big[\big( N(2p_1 + p_2) + M(2q_1 + q_2) \big)^2 + 3N(2p_1 + p_2) + 3M(2q_1 + q_2)\Big]\\
	&\ \ \ + \frac{NM}{(N+M)^2} \E(X-Y)^2 \\
A_3 & =  \frac{1}{(N+M)^2} \big[ (4N^2 + 6NM + 4M^2)^{\frac 14} + p_1 + p_2 + q_1 + q_2 \big]^2 \big[\sqrt{N+M} + p_1 + p_2 + q_1 + q_2\big].
}
Furthermore, set $M = m N$, $p_1 = \frac{\hat p_1}{N}$, $p_2 = \frac{\hat p_2}{N}$, $q_1 = \frac{\hat q_1}{N}$, $q_2 = \frac{\hat q_2}{N}$ and for some $\eps \in [0,1]$, $c = \hat c N^\eps$. Then,
\ba{
\E (X - Y)^2 = \frac{(\hat p_1 + m \hat q_1)(\hat p_2 + m \hat q_2)} { \hat c \big( \hat p_1 + m \hat q_1 + \hat p_2 + m \hat q_2\big)\big( 1 + 2(\hat p_1 + m \hat q_1) + 2(\hat p_2 + m \hat q_2)\big)}N^{-\eps} + \lito(N^{-\eps}).
}
\end{theorem}

\begin{theorem}\label{thm:SB}
Let $(X,Y)$ be distributed as the proportions of type $1$ for the stationary distribution of the finite population seed-bank model, and $(Z_1,Z_2) \sim \TI$ with parameters $a_1 = 2p_1(N-c)$, $a_2 = 2p_2(N-c)$, $b_1 = b_2=0$, $c_1 = 2c$, $c_2 = \frac{2cN}{M}$, $\alpha = 2$ and $\beta = 0$. Then for $h(x,y) = x^{n}y^{m}$,
\ba{
\abs{\E h(X,Y) - \E h(Z_1,Z_2)}&\leq D_y A_y + D_{xx}A_{xx} + D_{yy}A_{yy} + D_{xy}A_{xy} \\
	&\hspace{0.5cm}+ D_{xxx}A_{xxx} + D_{xxy}A_{xxy} + D_{xyy}A_{xyy} + D_{yyy}A_{yyy},
}
where
\ba{
A_y &=\frac{2cN}{M}(p_1 + p_2),\\
A_{xx} &= N (p_1+p_2)^2 + Np_1^2+(4c+3 + (4c+2+2N)p_1)(p_1+p_2) + (2c+3)p_1\\
	&\ \ \ \ +  \frac{1}{N}(3c^2 + 2c^2p_1 + cp_2) ,\\
A_{yy} &= \frac{4c^2N(1 + p_1)}{M(M-1)},\\
A_{xy} &=\frac{2}{M} \Big[ (2cN + cNp_1)(p_1+p_2) + (2c^2 + 2cN)p_1 + 3c^2+cNp_1^2 \frac{c^2M}{M-1}\Big],\\
A_{xxx} &=\frac16\Bigg[\frac1N\Big(4N^2 + 2cN + \eps_{M,c} \Big)^\frac14 + p_1+p_2+\frac{c}{N}\Bigg]^2\Bigg[ \sqrt N + N(p_1+p_2)+c \Bigg],\\
A_{xxy} &=\frac12\Bigg[\frac1N\Big( 4N^2 + 2cN + \eps_{M,c}  \Big)^\frac14 + p_1+p_2+\frac{c}{N}\Bigg]^2 \Bigg[\sqrt{ \frac{5cN^2}{4M^2}} + \frac{cN}{M}(1 + p_1 + p_2)\Bigg],\\
A_{xyy} &=\frac12\Bigg[ \frac{1}{M} \Big(6c^2 + \eps_{M,c}\Big)^\frac14 + \frac{c}{M}(1 + p_1 + p_2) \Bigg]^2 \Bigg[\sqrt N + N(p_1+p_2)+c \Bigg],\\
A_{yyy} &= \frac16\Bigg[ \frac{1}{M} \Big(6c^2 + \eps_{M,c}\Big)^\frac14 + \frac{c}{M}(1 + p_1 + p_2) \Bigg]^2 \Bigg[\sqrt{ \frac{5cN^2}{4M^2}} + \frac{cN}{M}(1 + p_1 + p_2))\Bigg],\\
\eps_{M,c} &=  \frac{cM}{4(M-1)(M-2)(M-3)} \Big[M^2(1+c) + M(1 + 6c + c^2) + 6c^2\Big],
}
and $D_x, D_y$, etc.~are defined in Theorem~\ref{thm:fcts}.
\end{theorem}
\begin{theorem}\label{thm:SBB}
Let $(X,Y)$ be distributed as the proportions of type $1$ for the stationary distribution of the finite population seed-bank model, and $Z \sim \text{Beta}(a_1, a_2)$ with parameters $a_1 = 2(N+M)p_1$ and $a_2 = 2(N+M)p_2$. Then for any $h \in \BC^{2,1}([0,1])$,
\ba{
\Bigg| \E h\Bigg( \frac{N}{N+M} X + \frac{M}{N+M}Y\Bigg) - \E h(Z) \Bigg| \leq \frac{|h|_1}{a_1 + a_2} A_1 + \frac{|h|_2}{2(a_1 + a_2 + 1)}A_2 + \frac{|h|_{2,1}}{18(a_1 + a_2 + 2)} A_3,
}
where
\ba{
A_1 &=2M(p_1+p_2) \sqrt{ \E(X - Y)^2}, \\
A_2 &= N(p_1 + p_2)^2 + (p_1 + p_2)+ \frac{M}{N+M}\sqrt{\E (X - Y)^2},\\
A_3 & =  \frac{2}{N(N+M)}\Big[(2N)^\frac12 + N(p_1 + p_2)\Big]^2 \Big[N^\frac12 + N(p_1 + p_2)\Big].
}
Furthermore, set $M = m N$, $p_1 = \frac{\hat p_1}{N}$, $p_2 = \frac{\hat p_2}{N}$, and for some $\eps \in [0,1]$, $c = \hat c N^\eps$. Then,
\ba{
\E (X - Y)^2 = \frac{m\hat p_1 \hat p_2} { \hat c (\hat p_1 +\hat p_2)\big( 1 + 2(m + 1)(\hat p_1 + \hat p_2) \big)}N^{-\eps} + \lito(N^{-\eps}).
}
\end{theorem}
\begin{remark}
In Theorems~\ref{thm:Beta} and~\ref{thm:SBB}, the terms $\E(X-Y)^2$ can be computed exactly using elementary methods, but the exact solutions are too lengthy to display, hence we provide only the leading terms. 
\end{remark}
Theorems~\ref{thm:WFs} and~\ref{thm:SBs} follow directly from the above 4 theorems. The results show that that for both the two-island and seed bank models, the beta approximations are appropriate when migration is strong ($\eps > 0$), and not when migration is weak ($\eps = 0$). 
When $\eps = 0$, it is clear that a two-island approximation is suitable and the beta approximation is not. However, the case where $\eps \in (0,1/2]$ is less clear. According to Theorems~\ref{thm:WFs} and~\ref{thm:SBs}, when $\eps \in (0,1/2]$, both approximations will converge to 0. While it is not surprising that there is a range of values of $\eps$ where the bound converges to zero given the bound is $\bigo(N^{-1/2})$ when $\eps = 0$, we take some space to explain why both approximations appear to be accurate. 

From our choice of parameters in Theorems~\ref{thm:WF} and~\ref{thm:SB}, if $c = \bigo(N^\eps)$, then when $\eps > 0$, the approximating TI distribution will have $c_1$ and $c_2$ are very large compared to the other parameters. The following lemma shows that the two islands become essentially indistinguishable from each other.

\begin{lemma}\label{lem:TIbeta}
\sloppy For the stationary distribution of the two-island diffusion, let $(X,Y) \sim \TI(a_1, a_2, b_1, b_2, c, \gamma c, \alpha,\beta)$. Then as $c \to \infty$, $X - Y \stackrel{p}{\to} 0$. Analogously for the stationary distribution of the seed-bank diffusion, let $(X,Y) \sim \TI(a_1, a_2, 0, 0, c, \gamma c, \alpha, 0)$. Then as $c \to \infty$, $X - Y \stackrel{p}{\to} 0$.
\end{lemma}

In particular, in light of Theorems~\ref{thm:Beta} and~\ref{thm:SBB}, we also present the following conjecture that under the following regime, the limiting distribution of the individual islands themselves are also beta distributed. Using the moment calculations of $X$ and $Y$ in the proof of Lemma~\ref{lem:TIbeta}, it is straightforward to see that the first and second moments of $X$ and $Y$ both match the moments of the beta distribution. 

\begin{conjecture}
For the stationary distribution of the two-island diffusion, let $(X,Y) \sim \TI(a_1, a_2, \gamma b_1, \gamma b_2, c, \gamma c, 2,2\gamma)$. Then as $c \to \infty$, both $X$ and $Y$ converge in distribution to the Beta distribution with parameters $(a_1 + b_1, a_2 + b_2)$. Analogously for the stationary distribution of the seed-bank diffusion, let $(X,Y) \sim \TI(a_1, a_2, 0, 0, c, \gamma c, 2, 0)$. Then as $c \to \infty$, both $X$ and $Y$ converge in distribution to the Beta distribution with parameters $(a_1, a_2)$.
\end{conjecture}

\section{Proofs of main results}

The primary tool used in this paper to prove the main results is Stein's method. Stein's method is a powerful tool in probability theory that is used to derive an \emph{explicit bound} for the difference between two probability distributions. Typically one aims to use it to find an upper bound on the errors incurred when approximating an intractable \emph{target} distribution with a commonly used simple \emph{reference} distribution. It was first developed for the Normal distribution in~\cite{stein72} to bound the approximation errors when applying the central limit theorem, and it has since been developed numerous distributions, such as Poisson~\cite{Chen75,BHJ}, beta~\cite{GR13, Dobler2015}, Dirichlet~\cite{GRR17}, negative binomial~\cite{BP99}, exponential~\cite{FulmanRoss2013} to just name a few. For many more examples and applications, see for example the surveys or monographs~\cite{Ross11, Chatterjee2014,introstein,LRS2017}. We first briefly introduce Stein's method and outline the main steps for the proofs of the main theorems.

\subsection{Introduction to Stein's method}\label{stnintro}
To successfully apply Stein's method, there are typically three main steps which we outline and give a brief description of our particular approach below. Our goal is to bound the difference between the typically unknown or intractable law of our target random variable $W$ with the law of a well understood and simple reference random variable $Z$.
\begin{enumerate}
\item Identify a characterising operator $\mathcal A$ or identity that is satisfied by only the reference distribution $\pi$. In this paper the characterising operators are the Wright-Fisher diffusion generator and the two-island diffusion generator. The generator characterises its associated stationary distribution through the identity that $\E \mathcal Af(Z) = 0$ for all suitable functions $f$ when $Z$ follows the stationary distribution.
\item For any arbitrary function $h$, solve for the function $f_h$ that satisfies
\ba{
\mathcal A f_h(x) = h(x) - \E h(Z).
}
Then by setting $x = W$ and taking expectations,
\ban{
\abs { \E \mathcal A f_h(W)} = \abs{ \E h(W) - \E h(Z)}.\label{eq:stneqa}
}
Properties of the function $f_h$ turn out to be crucial to derive a good bound with Stein's method. Typically one will require good bounds on $f_h$ and its derivatives. Much of this paper is dedicated to this process.
\item The final step is to bound~\eq{eq:stneqa} for all $h$ from a rich enough family of test functions $\mathcal H$, for example $\mathcal H$ could be the family of all polynomials, in which case would yield a bound on the differences in all moments for the two distributions. Rather than directly bounding $\abs{ \E h(W) - \E h(Z)}$, using~\eq{eq:stneqa} we instead derive the bound for $\abs { \E \mathcal A f_h(W)}$ which is more tractable. Our technique of choice in this paper is to use what is known in the Stein's method literature as the `exchangeable pairs' approach. 
\end{enumerate}

For Theorems~\ref{thm:Beta} and~\ref{thm:SBB}, steps 1 and 2 have been established in~\cite{GRR17} and we only need to complete the third step. However for the TI distribution, steps 1 and 2 have not been established in the literature, so we first derive its general Stein's method framework in this paper.

\subsection{Stein's method for the TI distribution}
In this section, we focus on the first two of the three steps outlined in Section~\ref{stnintro}, and leave the final step for the appendix. The first step for Stein's method for the stationary moments of the two-island diffusion is to define our characterising operator. We follow the generator approach of~\cite{B88, Barbour1990, Gotze1991} and our operator is therefore the generator of the two-island diffusion~\eq{eq:gen}.
\begin{lemma}
Given non-negative constants $a_1, a_2, b_1, b_2, c_1, c_2, \alpha, \beta$, let $(X,Y)$ be a random vector defined on $[0,1]^2$. Then $(X,Y) \sim \TI$ if and only if for all functions $f \in \mathrm{C}^2[0,1]^2$, 
\ba{
\E \cA f(X,Y) = 0,
}
where $\cA f(x,y)$ is defined as in~\eq{eq:gen}.
\end{lemma}
\begin{proof}
\cite[Proposition~2.6]{Blathetal2019} show that the two-island diffusion model has a unique stationary distribution so the forwards implication follows. The reverse implication can be shown by taking polynomial test functions to yield the moments.
\end{proof}

Recall the next step in Stein's method is for any arbitrary function $h$, to solve for the function $f_h$ that satisfies
\ban{
\cA f_h(x,y) = h(x,y) - \IE h(Z_1,Z_2).\label{eq:steineq}
}
It then follows that
\ba{
\abs{ \E h(X,Y) - \E h(Z_1, Z_2)} = \abs{\E \cA f_h(X,Y)},
}
and our aim is now to bound the right hand side of the above equation $\abs{ \E \cA f_h(X,Y)}$. Using the generator approach of~\cite{Barbour1990}, we will show that the solution $f_h$ to the equation~\eq{eq:steineq} is
\ban{
f(x,y) := f_h(x,y) = -\int_0^\infty [\E h((X,Y)_{(x,y)}(t)) - \E h(Z_1, Z_2)] dt,\label{eq:steinsol}
}
where $(X,Y)_{(x,y)}(t)$ is a diffusion process governed by generator~\eq{eq:gen} and starting value $(X,Y)_{(x,y)}(0) = (x,y)$. 
To successfully apply Stein's method, good bounds on the supremum norms of derivatives of~\eq{eq:steinsol}, known in the Stein's method literature as Stein factors, are essential. In particular, for our application we require bounds on the first three derivatives.

For a function $f: [0,1]^2 \to \IR$, we denote the supremum norm $\| f \|_\infty = \sup_{(x,y) \in [0,1]^2} \abs{f(x,y)}$, and the supremum bounds on the derivatives of $f$,
\ba{
|f|_x &= \left\| \frac{\partial}{\partial x}f(x,y) \right\|_\infty,\ |f|_y = \left\| \frac{\partial}{\partial y}f(x,y) \right\|_\infty.
}
Analogous notation for higher order derivatives, for example $|f|_{xx}$ will also be used. 
\begin{theorem}\label{thm:fcts}
For any test function of the form $h(x,y) = x^{n}y^{m}$ where $n, m$ are non-negative integers and set $a = a_2 + a_2$, $b = b_1 = b_2$. Then the function $f$ defined in~\eq{eq:steinsol} is the solution to~\eq{eq:steineq}, is well defined and the following bounds hold.
\ba{
|f|_{x} &\leq D_x = \frac{n(b+c_2) + mc_2}{ab + ac_2 + bc_1},\\
|f|_{y} &\leq D_y = \frac{m(a+c_1) + nc_1}{ab + ac_2 + bc_1},\\
|f|_{xy} &\leq D_y = \frac{m(a+c_1) + nc_1}{ab + ac_2 + bc_1},\\
|f|_{xx} &\leq D_{xx} = \frac{n(n-1)[b(a+b+ c_1 + c_2) + c_2(a+b + c_2)] + m(m-1)c_2^2 + 2mnc_2(b+c_2)}{2(a + b + c_1 + c_2)(ab + ac_2 + bc_1)},\\
|f|_{yy} &\leq D_{yy} =\frac{m(m-1)[a(a+b+ c_1 + c_2) + c_1(a+b + c_1)] + n(n-1)c_1^2 + 2mnc_1(a+c_1)}{2(a + b + c_1 + c_2)(ab + ac_2 + bc_1)},\\
|f|_{xy} &\leq D_{xy} = \frac{n(n-1)c_1(b+c_2) + m(m-1)c_2(a+c_1) + 2mn(a+c_1)(b+c_2)}{2(a + b + c_1 + c_2) (ab + ac_2 + bc_1)},\\
|f|_{xxx} &\leq D_{xxx} = \Big\{ n(n-1)(n-2) \big[b[ 2(a + b + c_1 + c_2)^2 + ab + bc_1 + c_1c_2] + c_2[ 2(a + b + c_2)^2\\
	&\hspace{2cm}+ a(2b + 2c_1 + c_2)]\big]  + 3mn(n-1) c_2(ab + ac_2 + bc_1 + 2(b+c_2)^2)\\
	&\hspace{2cm} + 6nm(m-1)c_2^2(b+c_2) + 2m(m-1)(m-2)c_2^3\Big\} \\
	&\hspace{2cm}\Big/\Big\{ 3 (ab + ac_2 + bc_1) (2(a + b + c_1 + c_2)^2 + ab + ac_2 + bc_1)\Big\},\\
|f|_{xxy} &\leq D_{xxy} = \Big\{ n(n-1)(n-2) c_1(ab + ac_2 + bc_1 + 2(b+c_2)^2) + 4mn(n-1) c_1c_2(b+c_2) \\
	&\hspace{2cm} + 2nm(m-1)c_1c_2^2 + mn(n-1) \big[ 3 b (a + c_1) (a + 2 b + c_1) \\
	&\hspace{2cm}+(3 a^2 + 8 b c_1 + 3 a (4 b + c_1)) c_2 + 2 (3 a + c_1) c_2^2\big] \\
	&\hspace{2cm}+ 2nm(m-1)c_2(3a(b+c_2) + c_1(3b+2c_2)) + 2m(m-1)(m-2)(a+c_1)c_2^2 \Big\}\\
	&\hspace{2cm} \Big/\Big\{ 3 (ab + ac_2 + bc_1) (2(a + b + c_1 + c_2)^2 + ab + ac_2 + bc_1)\Big\}\\
|f|_{xyy} &\leq D_{xyy} = \Big\{ m(m-1)(m-2) c_2(ab + ac_2 + bc_1 + 2(a+c_1)^2) + 4nm(m-1) c_1c_2(a+c_1) \\
	&\hspace{2cm} + 2mn(n-1)c_1c_2^2 + nm(m-1) \big[ 3 a (b + c_2) (b + 2 a + c_2) \\
	&\hspace{2cm}+(3 b^2 + 8 a c_2 + 3 b (4 a + c_2)) c_1 + 2 (3 b + c_2) c_1^2\big] \\
	&\hspace{2cm}+ 2mn(n-1)c_1(3b(a+c_1) + c_2(3a+2c_1)) + 2n(n-1)(n-2)(b+c_2)c_1^2 \Big\}\\
	&\hspace{2cm} \Big/\Big\{ 3 (ab + ac_2 + bc_1) (2(a + b + c_1 + c_2)^2 + ab + ac_2 + bc_1)\Big\}\\
|f|_{yyy} &\leq D_{yyy} = \Big\{ m(m-1)(m-2) \big[a[ 2(a + b + c_1 + c_2)^2 + ab + ac_2 + c_1c_2] + c_1[ 2(a + b + c_1)^2\\
	&\hspace{2cm}+ b(2a + 2c_2 + c_1)]\big]  + 3nm(m-1) c_1(ab + ac_2 + bc_1 + 2(a+c_1)^2)\\
	&\hspace{2cm} + 6mn(n-1)c_1^2(a+c_1) + 2n(n-1)(n-2)c_1^3\Big\} \\
	&\hspace{2cm}\Big/\Big\{ 3 (ab + ac_2 + bc_1) (2(a + b + c_1 + c_2)^2 + ab + ac_2 + bc_1)\Big\}.
}
\end{theorem}

Before we prove Theorem~\ref{thm:fcts}, we describe a key component of the proofs, moment duality of the two-island diffusion model. For a more detailed discussion of this duality see~\cite{Blathetal2019}. For $f(x,y) = x^n y^{m}$, where $n, m \in \IZ_+$, then~\eq{eq:gen} takes the form
\ba{
\cA f(x,y) &= \frac\alpha2 x(1-x) n(n-1) x^{n-2}y^m + [ a_1 - (a_1+a_2) x + c_1(y-x) ]nx^{n-1}y^{m}\\
	&\ \ + \frac\beta2 y(1-y) m(m-1) x^{n}y^{m-2} + [ b_1 - (b_1+b_2) y + c_2(x-y) ]mx^{n}y^{m-1}\\
	&= \Big(\frac\alpha2n(n-1)+ na_1\Big) \Big[ x^{n-1} y^m - x^ny^m\Big] + na_2 \Big[0 - x^ny^m\Big] + nc_1\Big[x^{n-1}y^{m+1} - x^ny^m\Big]\\
	&\ \ +\Big(\frac\beta2m(m-1) + mb_1\Big) \Big[ x^{n} y^{m-1} - x^ny^m\Big] + mb_2\Big[0 - x^ny^m\Big] + mc_2\Big[x^{n+1}y^{m-1} - x^ny^m \Big].
}
We define our dual process on the space of polynomial functions. We could equivalently define the process on tuples $(n,m)$ as in~\cite[Definition~2.1]{Blathetal2019}, but in our case it is more natural to work in the space of polynomial functions. Given the function is in `state' $x^ny^m$ it transitions to the following with rates as described:
\begin{itemize}
\item $x^{n-1}y^m$ with rate $\frac\alpha2 n(n-1) + na_1$,
\item $x^{n-1}y^{m+1}$ with rate $c_1n$,
\item $x^ny^{m-1}$ with rate $\frac\beta2 m(m-1) + mb_1$,
\item $x^{n+1}y^{m-1}$ with rate $c_2m$,
\item $0$ with rate $na_2 + mb_2$.
\end{itemize}
Let $Q^{x,y}_{n,m}(t)$ be the process governed by the rates as above and with initial value $Q^{x,y}_{n,m}(0) = x^ny^m$. 
\begin{lemma}\label{lem:dual}
For functions $h$ of the form $h(x,y) = x^ny^m$,
\ba{
\E h((X,Y)_{x,y}(t)) = \E Q^{x,y}_{n,m}(t).
}
\end{lemma}
\begin{proof}
This follows from the dual process construction and~\cite[Proposition~1.2]{JK14}.
\end{proof}
\begin{proof}[Proof of Theorem~\ref{thm:fcts}]
Using Lemma~\ref{lem:dual},
\ba{
f(x,y) = -\int_0^\infty ]\E Q^{x,y}_{n,m}(t) - \E h (Z_1,Z_2)]dt.
}
The process $Q^{x,y}_{n,m}(t)$ will reach an absorbing state in finite time, and this is straightforward by considering process of the sum of the exponents and noting that it is a non-increasing process. Furthermore~\cite[Propositions 2.4, 2.6]{Blathetal2019} show that the expected value of the process at absorption is exactly $\E h(Z_1,Z_2)$. Hence $f(x,y)$ is well-defined.

Again using Lemma~\ref{lem:dual}, it is straightforward to see that the semi-group generated by $\cA$ is strongly continuous. We then use~\cite[Proposition~1.5(a)]{EK86}, which implies $f^{(u)}(x,y) := -\frac12 \int_0^u[ \E h((X,Y)_{(x,y)}(t)) - \E h (Z_1,Z_2)] dt$ is in the domain of $\cA$ and is the solution to
\ba{
\cA f^{(u)}(x,y) = h(x,y) - \E h((X,Y)_{(x,y)}(u)).
}
Furthermore,~\cite[Corollary 1.6]{EK86} implies that $\cA$ is a closed operator, and as $u \to \infty$, again using~~\cite[Propositions 2.4, 2.6]{Blathetal2019},
\ba{
\sup_{(x,y) \in [0,1]^2}\abs{ \cA f^{(u)}(x,y) - \cA f(x,y) } = \sup_{(x,y) \in [0,1]^2}\abs{\E h(Z_1,Z_2) - \E h( (X,Y)_{x,y}(u) )} \to 0.
}
Hence~\eq{eq:steineq} is satisfied by $f(x,y)$.

For the first derivative,
\ba{
\frac{\partial}{\partial x}f(x,y) &= \lim_{\eps \to 0} \frac1\eps\left[ f(x + \eps, y) - f(x,y)\right]\\
	&= \lim_{\eps \to 0} -\frac1\eps \int_0^\infty \left[\E h((X,Y)_{(x+\eps,y)}(t)) - \E h((X,Y)_{(x,y)}(t))\right] dt.
}
Using the duality of the semi-group,
\ba{
\bbabs{\frac{\partial}{\partial x}f(x,y)} &= \left| \lim_{\eps \to 0} \frac1\eps \int_0^\infty \left[\E h((X,Y)_{(x+\eps,y)}(t)) - \E h((X,Y)_{(x,y)}(t))\right] dt\right|\\
	&\leq  \lim_{\eps \to 0}\frac1\eps \int_0^\infty \E \left| Q^{x+\eps,y}_{n,m}(t) - Q^{x,y}_{n,m}(t) \right| dt.
}
We can couple the transition times of the two processes $Q^{x+\eps,y}_{n,m}(t)$, $Q^{x,y}_{n,m}(t)$ exactly in the obvious manner, and the two processes will become identical when they reach one of the two absorbing states. We first note that
\ba{
\abs{ (x+\eps)^{k_1} y^{k_2} - x^{k_1}y^{k_2} } &= \eps k_1 x^{k_1-1}y^{k_2} +\lito(\e) \\
	&\leq \eps k_1 + \lito(\e).
}
It is now straightforward to see that $\abs{Q^{x+\eps,y}_{n,m}(t) - Q^{x,y}_{n,m}(t)} \leq (n+m)\eps + \lito(\e)$ for all $t \leq \tau_{n,m}$ where $\tau_{n,m}$ is the absorption time of the process. Therefore
\ban{
\int_0^\infty \E \left| Q^{x+\eps,y}_{n,m}(t) - Q^{x,y}_{n,m}(t) \right| dt = \eps \int_0^\infty \E N^*_{n,m}(t) dt+\lito(\e),\label{eq:f1}
}
where $N^*_{n,m}(t)$ is the exponent of $x$ at time $t$ and the starting exponents of $x$ and $y$ are $n$ and $m$ respectively. To derive a bound for the above integral, we define the following auxiliary urn process. Consider two urns, and suppose there are $n$ balls in the first urn and $m$ balls in the second urn. 
\begin{itemize}
\item At rate $(a_1 + a_2)n$, remove a ball from the first urn.
\item At rate $(b_1 + b_2)m$, remove a ball from the second urn.
\item At rate $c_1n$, move a ball from the first to second urn.
\item At rate $c_2m$, move a ball from the second to first urn. 
\end{itemize}
Let $N_{n,m}(t)$ (and $M_{n,m}(t)$) denote the number of balls in the first (second) urn at time $t$ given we start with $n$ ($m$) balls in the first (second) urn. Given the rates at which balls are removed are slower than the rates of the dual process, it is clear we can construct a coupling such that $N^*_{n,m}(t)$ is stochastically dominated by $N_{n,m}(t)$. We begin by focusing on the case $n=1$ and $m=0$ and the partial derivative with respect to $x$. We also simultaneously solve the case $n=0$ and $m=1$. Let $P_{n,m,t}(z_1, z_2) = \E \big(z_1^{N_{n,m}(t)}z_2^{M_{n,m}(t)}\big)$ denote the joint probability generating function for $(N_{n,m}(t), M_{n,m}(t))$ and set $a = a_1 + a_2$ and $b = b_1 + b_2$. Then routine calculations (similar to the derivation of Kolmogorov forwards/backwards equations) yield,
\ba{
\frac{\partial}{\partial t}P_{1,0,t}(z_1,z_2)=a + c_1P_{0,1,t}(z_1,z_2) - (a+c_1)P_{1,0,t}(z_1,z_2).
}
In an analogous calculation,
\ba{
\frac{\partial}{\partial t}P_{0,1,t}(z_1,z_2)=b + c_2P_{1,0,t}(z_1,z_2) - (b+c_2)P_{0,1,t}(z_1,z_2).
}
The explicit solutions to the above pair of differential equations are elementary to compute, however they are very long so we omit them. We note that one could equivalently solve the standard Kolmogorov forwards or backwards equations to calculate the transition probabilities and expectations without any difficulty. Both approaches yield the same final result, but we chose this approach since the joint nature of the probability generating function results in fewer equations. From the properties of probability generating functions, we can explicitly compute
\ba{
\E N_{1,0}(t) &= \frac{\partial}{\partial z_1}P_{1,0,t}(z_1,z_2)\Big|_{z_1=z_2=1},\\
\E N_{0,1}(t) &=  \frac{\partial}{\partial z_1}P_{0,1,t}(z_1,z_2)\Big|_{z_1=z_2=1}.
}
Recalling that our goal is to bound~\eq{eq:f1} (for $n=1, m = 0$),
\ban{
\int_0^\infty \E N_{1,0}(t) dt &= \frac{b+c_2}{ab + ac_2 + bc_1}.\label{eq:f6}
}
And we can now also solve the case where $n=0$ and $m=1$,
\ban{
\int_0^\infty \E N_{0,1}(t) dt &= \frac{c_2}{ab + ac_2 + bc_1}.\label{eq:f7}
}
To solve for general $n$ and $m$, for any time $t$, we can represent $N_{n,m}(t)$ as the number of balls currently in the first urn which originated from the first urn, and the number of balls in the first urn that originated from the second urn. From our construction, we can treat each ball independently. Set $N_{n,m}(t) = N_{n,0}(t) + N_{0,m}(t)$, where $N_{n,0}(t) \sim \Bin(n, \E N_{1,0}(t))$ and $N_{0,m}(t) \sim \Bin(m, \E N_{0,1}(t))$ and $N_{n,0}(t)$ and $N_{0,m}(t)$ are independent. Hence $\E N_{n,m}(t) = n \E N_{0,1}(t) + m \E N_{0,1}(t)$. Also definite analogous variables for $M_{n,m}(t)$. Hence using~\eq{eq:f1},~\eq{eq:f6} and~\eq{eq:f7},
\ba{
\bbabs{\frac{\partial}{\partial x}f(x,y)} &\leq \int_0^\infty \E N_{n,m}(t) dt= \frac{n(b+c_2) + mc_2}{ab + ac_2 + bc_1},\\
\bbabs{\frac{\partial}{\partial y}f(x,y)} &\leq \int_0^\infty \E M_{n,m}(t) dt= \frac{m(a+c_1) + nc_1}{ab + ac_2 + bc_1},
}
where the second equation follows from symmetry.

For the second derivatives, we start with the case of $\frac{\partial^2}{\partial x^2} f(x,y)$. First,
\ba{
\bbabs{ \frac{\partial^2}{\partial x^2} f(x,y)} \leq \lim_{\eps_1,\eps_2 \to 0} \frac{1}{\eps_1\eps_2} \int_0^\infty \E \abs{ Q_{n,m}^{x+\eps_1+\eps_2,y}(t) - Q_{n,m}^{x+\eps_1,y}(t)-Q_{n,m}^{x+\eps_2,y}(t)+Q_{n,m}^{x,y}(t) }dt,
}
and 
\ba{
\abs{ (x+\eps_1 + \eps_2)^{k_1}y^{k_2} - (x+\eps_1)^{k_1}y^{k_2} - (x+\eps_2)^{k_1}y^{k_2} + x^{k_1}y^{k_2} } =  \eps_1\eps_2 k_1(k_1-1)x^{k_1-2}y^{k_2} +\lito(\e_1\e_2).
}
Therefore using factorial moments of the binomial distribution,
\ban{
\bbabs{ \frac{\partial^2}{\partial x^2} f(x,y)} &\leq \int_0^\infty \E [N_{n,m}(t)(N_{n,m}(t) -1)] dt\notag\\
	&\leq \int_0^\infty \E[ (N_{n,0}(t) + N_{0,m}(t))(N_{n,0}(t) + N_{0,m}(t) -1 )] dt\notag\\
	&= n(n-1)\int_0^\infty [\E N_{1,0}(t)]^2 dt + m(m-1) \int_0^\infty [\E N_{0,1}(t)]^2 dt \notag\\
	&\hspace{1cm}+ 2nm \int_0^\infty  [\E N_{0,1}(t)][\E N_{1,0}(t)] dt. \label{eq:f5}
}
Treating the three integrals separately for readability,
\ban{
\int_0^\infty [\E N_{1,0}(t)]^2 dt &= \frac{b(a+b+ c_1 + c_2) + c_2(a+b + c_2)}{ 2(a+b + c_1 + c_2)(ab + ac_2 + bc_1 )}.\label{eq:f2}\\
\int_0^\infty [\E N_{0,1}(t)]^2 dt&=  \frac{c_2^2}{ 2(a+b + c_1 + c_2)(ab + ac_2 + bc_1)}.\label{eq:f3}\\
\int_0^\infty [\E N_{0,1}(t)][\E N_{1,0}(t)] dt &= \frac{c_2(b+c_2)}{2(a+b + c_1 + c_2)(ab + ac_2 + bc_1)}.\label{eq:f4}
}
Combining~\eq{eq:f2},~\eq{eq:f3} and~\eq{eq:f4} in~\eq{eq:f5} yields the final bound.
For the mixed second derivative we begin with noting
\ba{
\abs{ (x+\eps_1)^{k_1}(y+\eps_2)^{k_2} - (x+\eps_1)^{k_1}y^{k_2} - x^{k_1}(y+\eps_2)^{k_2} + x^{k_1}y^{k_2} } =  \eps_1\eps_2 k_1k_2x^{k_1-1}y^{k_2-1} +\lito(\e_1\e_2).
}
Then,
\ban{
\bbabs{\frac{\partial^2}{\partial x \partial y} f(x,y)} & \leq \int_0^\infty \E[ N_{n,m}(t) M_{n,m}(t)] dt\notag\\
	&= \int_0^\infty \E \big[(N_{n,0}(t) + N_{0,m}(t))(M_{n,0}(t) + M_{0,m}(t))\big] dt\notag\\
	&= \int_0^\infty \big[n(n-1)\E N_{1,0}(t) \E M_{1,0}(t) + m(m-1) \E N_{0,1}(t)\E M_{0,1}(t)\notag\\
	&\hspace{1cm} + nm\E N_{1,0}(t) \E M_{0,1}(t) + nm\E N_{0,1}(t) \E M_{1,0}(t) \big] dt,\label{eq:f8}
}
where we have used the fact that $(N_{n,0}(t), M_{n,0}(t))$ is multinomially distributed with probabilities $\E N_{0,1}(t)$ and $\E M_{0,1}(t)$. Computations then yield
\ban{
\int_0^\infty [\E N_{1,0}(t) \E M_{1,0}(t)]dt &=\frac{c_1(b+c_2)}{2(a + b + c_1 + c_2) (ab + ac_2 + bc_1)},\label{eq:f9}
}
and by symmetry
\ban{
\int_0^\infty [\E N_{0,1}(t)M_{0,1}(t)] dt = \frac{c_2(a+c_1)}{2(a + b + c_1 + c_2) (ab + ac_2 + bc_1)}.\label{eq:f10} 
}
And finally using the same formulas as in~\eq{eq:f6},~\eq{eq:f7},
\ban{
\int_0^\infty [  \E N_{1,0}(t) \E M_{0,1}(t) + \E N_{0,1}(t) \E M_{1,0}(t)] dt = \frac{2(a+c_1)(b+c_2)}{2(a + b + c_1 + c_2) (ab + ac_2 + bc_1)}.\label{eq:f11}
}
The bound now follows from using~\eq{eq:f9}, ~\eq{eq:f10} and~\eq{eq:f11} in~\eq{eq:f8}.

Before we begin the bounds for the third derivative, we note the following. There is a minor issue that the domain of $\cA$ is continuous functions with 2 derivatives, so it is not immediate that the third derivative of $f$ exists. However, we can utilise the approach used in~\cite[Lemmas~3.6, 3.7]{GRR17} to show that since our polynomial test functions $h$ have bounded derivatives of all orders, then the third derivative exists. 

As before
\ba{
|(x + \e_1 + &\e_2 + \e_3)^{k_1}y^{k_2} - (x + \e_1 + \e_2)^{k_1}y^{k_2} - (x + \e_1 + \e_3)^{k_1}y^{k_2} - (x + \e_2 + \e_3)^{k_1}y^{k_2} \\
	&+ (x + \e_1)^{k_1}y^{k_2} + (x + \e_2)^{k_1}y^{k_2} + (x + \e_3)^{k_1}y^{k_2} - x^{k_1}y^{k_2} | \\
	&\leq k_1(k_1-1)(k_1-2)\e_1\e_2\e_3 + \lito(\e_1\e_2\e_3).
}
Then
\ba{
\Big| \frac{\partial^3}{\partial x^3}& f(x,y)\Big| \leq \int_0^\infty \E\big[ N_{n,m}(t)(N_{n,m}(t) - 1) (N_{n,m}(t)-2)\big] dt\\
	&= \int_0^\infty \E\big[ (N_{n,0}(t) + N_{0,m}(t)) (N_{n,0}^{(1)}(t) + N_{0,m}^{(2)}(t)-1)(N_{n,0}(t) + N_{0,m}(t)-2)\big] dt\\
		&= n(n-1)(n-2) \int_0^\infty [\E N_{1,0}(t)]^3dt + 3n(n-1)m \int_0^\infty [\E N_{1,0}(t)]^2\E N_{0,1}(t) dt\\
		& \hspace{0.5cm} + 3nm(m-1) \int_0^\infty \E N_{1,0}(t)[\E N_{0,1}(t)]^2 dt + m(m-1)(m-2) \int_0^\infty [\E N_{0,1}(t)]^3dt.
}
We record the following,
\ba{
\int_0^\infty& [\E N_{1,0}(t)]^3 dt \\
	&= \frac{b[ 2(a + b + c_1 + c_2)^2 + ab + bc_1 + c_1c_2] + c_2[ 2(a + b + c_2)^2 + a(2b + 2c_1 + c_2)] }
   {3 (ab + ac_2 + bc_1) (2(a + b + c_1 + c_2)^2 + ab + ac_2 + bc_1)},\\
\int_0^\infty& [\E N_{1,0}(t)]^2\E N_{0,1}(t) dt \\
	&= \frac{c_2 ( ab + ac_2 + bc_1 + 2(b + c_2)^2)}{3 (ab + ac_2 + bc_1) (2(a + b + c_1 + c_2)^2 + ab + ac_2 + bc_1)},\\
\int_0^\infty  &\E N_{1,0}(t) [ \E N_{0,1}(t)]^2 dt\\
	&= \frac{2c_2^2(b+c_2)}{3 (ab + ac_2 + bc_1) (2(a + b + c_1 + c_2)^2 + ab + ac_2 + bc_1)},\\
\int_0^\infty  &[ \E N_{0,1}(t)]^3 dt\\
	&= \frac{2c_2^3}{3 (ab + ac_2 + bc_1) (2(a + b + c_1 + c_2)^2 + ab + ac_2 + bc_1)}.\\
}
The bound corresponding to $D_{xxx}$ now follows, as well as the bound for $D_{yyy}$ by symmetry. For the bound corresponding to $D_{xxy}$, again using
\ba{
\big| (x + \e_1 &+ \e_2)^{k_1}(y+\e_3)^{k_2} - (x + \e_1 + \e_2)^{k_1}y^{k_2} - (x+\e_1)^{k_1}(y+\e_3)^{k_2} + (x+\e_1)^{k_1}y^{k_2}\\
	&- (x+\e_2)^{k_1}(y+\e_3)^{k_2} + (x+\e_2)^{k_1}y^{k_2} + x^{k_1}(y+\e_3)^{k_2} - x^{k_1}y^{k_2}\big|\\
	&\leq k_1(k_1-1)k_2\e_1\e_2\e_3 + \lito(\e_1\e_2\e_3). 
}
Then
\ban{
\Big| &\frac{\partial^3}{\partial x^2\partial y} f(x,y) \Big| \leq \int_0^\infty \E\big[ M_{n,m}(t) N_{n,m}(t) (N_{n,m}(t) - 1) \big] dt\notag\\
	&= \int_0^\infty \E \big[(M_{n,0}(t) + M_{0,m}(t))(N_{n,0}(t) + N_{0,m}(t))(N_{n,0}(t) + N_{0,m}(t)-1 )\big]dt\notag\\
	&= \int_0^\infty \E \Big[ M_{n,0}(t) N_{n,0}(t)(N_{n,0}(t) - 1) + 2 M_{n,0}(t)N_{n,0}(t)N_{0,m}(t) + M_{n,0}(t) N_{0,m}(t)(N_{0,m}(t)-1)\notag\\
	&+ M_{0,m}(t) N_{n,0}(t)(N_{n,0}(t)-1) + 2M_{0,m}(t) N_{n,0}(t) N_{0,m}(t) + M_{0,m}(t) N_{0,m}(t)(N_{0,m}(t) - 1) \Big] dt\notag \\
	&= \int_0^\infty \Big[n(n-1)(n-2)\E M_{1,0}(t) [ \E N_{1,0}(t)]^2 + 2mn(n-1) \E M_{1,0}(t) \E N_{1,0}(t) \E N_{0,1}(t)\notag\\
	&\hspace{0.5cm} + nm(m-1) \E M_{1,0}(t) [\E N_{0,1}(t)]^2 + mn(n-1) \E M_{0,1}(t) [ \E N_{1,0}(t)]^2\notag \\
	&\hspace{0.5cm}+ 2nm(m-1)\E M_{0,1}(t) \E N_{1,0}(t) \E N_{0,1}(t) + m(m-1)(m-2) \E M_{0,1}(t) [\E N_{0,1}(t)]^2\Big]dt .\label{eq:f12}
}
We now record the following.
\ba{
\int_0^\infty[ \E & M_{1,0}(t)[\E N_{1,0}(t)]^2] dt\\
	&= \frac{c_1(ab + ac_2 + bc_1 + 2(b+c_2)^2)}{3 (ab + ac_2 + bc_1) (2(a + b + c_1 + c_2)^2 + ab + ac_2 + bc_1)},\\
\int_0^\infty[ \E& M_{1,0}(t)\E N_{1,0}(t)\E N_{0,1}(t)]dt \\
	&=\frac{2c_1c_2(b+c_2)}{3 (ab + ac_2 + bc_1) (2(a + b + c_1 + c_2)^2 + ab + ac_2 + bc_1)},\\
\int_0^\infty [\E& M_{0,1}(t)[\E N_{0,1}(t)]^2 dt \\
	&=\frac{2c_1c_2^2}{3 (ab + ac_2 + bc_1) (2(a + b + c_1 + c_2)^2 + ab + ac_2 + bc_1)},\\
\int_0^\infty [\E& M_{0,1}(t) [\E N_{1,0}(t)]^2 ]dt\\
	&=\frac{3 b (a + c_1) (a + 2 b + c_1) + (3 a^2 + 8 b c_1 + 3 a (4 b + c_1)) c_2 + 
 2 (3 a + c_1) c_2^2}{3 (ab + ac_2 + bc_1) (2(a + b + c_1 + c_2)^2 + ab + ac_2 + bc_1)},\\
\int_0^\infty [ \E &M_{0,1}(t) \E N_{1,0}(t) \E N_{0,1}(t)] dt\\
	&= \frac{c_2 (3 a (b + c_2) + c_1 (3 b + 2 c_2))}{3 (ab + ac_2 + bc_1) (2(a + b + c_1 + c_2)^2 + ab + ac_2 + bc_1)},\\
\int_0^\infty [ \E& M_{0,1}(t) [\E N_{0,1}(t)]^2] dt\\
	&= \frac{2(a+c_1)c_2^2}{3 (ab + ac_2 + bc_1) (2(a + b + c_1 + c_2)^2 + ab + ac_2 + bc_1)}.
}
The remaining bounds now follow from the above display and symmetry.
\end{proof}
\begin{remark}
A worthwhile comparison of the above bounds are to~\cite[Theorem~5]{GRR17}, which are in some sense Stein factors for the ``one''-island model. For a meaningful comparison, we set $c_1 = c_2=m =0$ and $a= b$. For $|f|_x$, $|f|_{xx}$ and $|f|_{xxx}$, Theorem~\ref{thm:fcts} yields $\frac{n}{a}$, $\frac{n(n-1)}{2a}$  and $\frac{n(n-1)(n-2)}{3a}$ respectively and \cite[Theorem~5]{GRR17} yields $\frac{n}{a}$, $\frac{n(n-1)}{2(a+1)}$ and $\frac{n(n-1)(n-2)}{3(a+2)}$. Hence our bounds are comparably close but not quite as sharp. Ultimately our bounds are not as sharp as our calculations are based upon the simplified urn model rather than the true dual process. Furthermore we have no dependence on $\alpha$ and $\beta$.
\end{remark}
\begin{remark}
In the previous theorem we have chosen to derive bounds that are simple, explicit, and easy to use. The true bounds are certainly sharper, and we outline below an approach that would give sharper bounds, at the cost of significant complexity. This approach combines moment duality and the coupling ideas of~\cite{BX01}. This will ultimately yield a recursive set of equations that will solve for all the Stein solutions. 
We begin with a base case of $h(x,y) = x$. As before
\ba{
\frac{\partial}{\partial x} f_{h(x,y) = x}(x,y) = \lim_{\eps \to 0}-\frac{1}{\eps} \int_0^\infty \E[ Q_{1,0}^{x+\eps,y}(t) - Q_{1,0}^{x,y}(t)] dt.
}
Let $\tau \sim \Exp(a + c_1)$ be the waiting time until the first transition of the processes, then using the strong Markov property,
\ba{
&\frac{\partial}{\partial x} f_{h(x,y) = x}(x,y) \\
&=\lim_{\eps \to 0} -\frac{1}{\eps}\Bigg[ \E \int_0^\tau \eps dt  + \frac{c_1}{a + c_1} \int_\tau^\infty  \E[ Q_{1,0}^{x+\eps,y}(t) - Q_{1,0}^{x,y}(t)| Q_{1,0}^{x+\eps,y}(\tau) = Q_{1,0}^{x,y}(\tau)=y ] dt\Bigg]\\
	&= -\frac{1}{a + c_1} + \frac{c_1}{a + c_1} \frac{\partial}{\partial x} f_{h(x,y) = y}(x,y).
}
Similarly one can show,
\ba{
\frac{\partial}{\partial x} f_{h(x,y)=y} = \frac{c_2}{b+c_2} \frac{\partial}{\partial x} f_{h(x,y) = x}(x,y).
}
Solving the two equations yields
\ba{
\frac{\partial}{\partial x} f_{h(x,y) = x}(x,y) = - \frac{b + c_2}{ab + bc_1 + ac_2}.
}
The higher moment test functions can be recursively solved for in the following manner for example. 
\ba{
\frac{\partial}{\partial x} f_{h(x,y) = x^2}(x,y) &=-\frac{2}{\a + 2a + 2c_1} + \frac{\a + 2a_1}{\a + 2a + 2c_1} \frac{\partial}{\partial x} f_{h(x,y) = x}(x,y)\\
&\hspace{1cm}+ \frac{2c_1}{\a + 2a + 2c_1} \frac{\partial}{\partial x} f_{h(x,y) = xy}(x,y).
}
\end{remark}

%%%%%%%%%%%%%%%%%%%%%%%%%%%%%%%%%%%%%%%%%%%%%%%
Theorems~\ref{thm:WF} and~\ref{thm:SB} are both derived using the following general approximation theorem for the stationary moments of the two-island diffusion model.
\begin{theorem}\label{thm:ex}
Let $(X,Y)$ and $(X',Y')$ be two random vectors on $[0,1]^2$ such that $\mathscr{L}(X,Y) = \mathscr{L}(X',Y')$. Suppose for $\lambda \in \IR$, $R_1$ and $R_2$ are random variables that satisfy
\ban{
\E [ X'-X | (X,Y) ] &= \lambda [(a_1 - (a_1 + a_2))X + c_1(Y-X)] +  R_1,\label{eq:exid}\\
\E [ Y'-Y | (X,Y) ] &= \lambda [(b_1 - (b_1 + b_2))X + c_2(Y-X)] +  R_2.\label{eq:exid1}
}
Then for $h(x,y) = x^ny^m$, and $(Z_1,Z_2)\sim \TI$,
\ba{
\abs{\E h(X,Y) - \E h(Z_1,Z_2)}&\leq D_x A_x + D_y A_y + D_{xx}A_{xx} + D_{yy}A_{yy} + D_{xy}A_{xy} \\
	&\hspace{0.5cm}+ D_{xxx}A_{xxx} + D_{xxy}A_{xxy} + D_{xyy}A_{xyy} + D_{yyy}A_{yyy},
}
where
\ba{
A_x &:= \frac{1}{\lambda} \E | R_1|,\\
A_y &:=  \frac{1}{\lambda} \E |R_2|,\\
A_{xx} &:= \E\Big| \frac\alpha2X(1-X)  - \frac{1}{2\l}\E[(X'-X)^2|(X,Y)] \Big|,\\
A_{yy} &:=  \E\Big| \frac\beta2Y(1-Y)  - \frac{1}{2\l}\E[(Y'-Y)^2|(X,Y)] \Big|,\\
A_{xy} &:= \frac{1}{\l} \E \Big| \E[(X'-X)(Y'-Y)|(X,Y)] \Big|,\\
A_{xxx} &:= \frac{1}{6\l} \E\babs{(X'-X)^3},\\
A_{xxy} &:= \frac{1}{2\l} \E\babs{(X'-X)^2(Y'-Y)},\\
A_{xyy} &:= \frac{1}{2\l} \E\babs{(X'-X)(Y'-Y)^2},\\
A_{yyy} &:= \frac{1}{6\l} \E\babs{(Y'-Y)^3},
}
and $D_x, D_y$, etc.~are defined in Theorem~\ref{thm:fcts}.
\end{theorem}
\begin{proof}
We begin with the bivariate Taylor expansion of $(x',y')$ around $(x,y)$,
\ba{
f(x',y')& - f(x,y)\\
	& = (x'-x)f_x(x,y) + (y'-y)f_y(x,y) \\
	&+ \frac12(x'-x)^2 f_{xx}(x,y) + \frac12 (y'-y)^2 f_{yy}(x,y) + (x'-x)(y'-y)f_{xy}(x,y)\\
	&+  \frac{1}{6\l} \sum_{i=0}^3{3 \choose i} \E[ (x'-x)^i(y'-y)^{3-i} f_3^{(i)}(x^*_i,y^*_i)],
}
for some collection of constants $\{(x^*_i, y^*_i)\}$ and $f_3^{(i)}$ denotes the third derivative with $i$ partial derivatives with respect to $x$ and $3-i$ with respect to $y$. Substituting in $X$ and $Y$, and taking the expectation of the above, using the assumptions~\eq{eq:exid},~\eq{eq:exid1} and since $(X',Y') \ed (X,Y)$ implies $\E f(X',Y') - \E f(X,Y) = 0$,
\ban{
0 &=  \E\{[a_1 - (a_1 + a_2)X + c_1(Y-X)] f_x(X,Y)\} + \E [b_1 - (b_1 + b_2)X + c_2(Y-X)]f_y(X,Y)\notag  \\
	&\ \ \ + \frac1\l\E R_1 f_x(X,Y) + \frac 1 \l \E R_2 f_y(X,Y)\notag\\
	&+  \frac1{2\l}\E[\E[(X'-X)^2|(X,Y)] f_{xx}(X,Y)] + \frac1{2\l} \E[\E[ (Y'-Y)^2|(X,Y)] f_{yy}(X,Y)]\notag \\
	&\ \ \ + \frac1{\l}\E[\E[(X'-X)(Y'-Y)|(X,Y)]f_{xy}(X,Y)]\notag\\
	&+ \frac{1}{6\l} \sum_{i=0}^3{3 \choose i} \E[ (X'-X)^i(Y'-Y)^{3-i} f_3^{(i)}(x^*_i,y^*_i)]. \label{eq:ex2}
}
Recall that our objective is to bound the following:
\ban{
\E \cA f(X,Y) &= \E \Big\{[ a_1 - (a_1 + a_2)X + c_1(Y-X)] f_x(X,Y) + \frac\alpha2X(1-X) f_{xx}(X,Y)\notag\\
	&\ \ \ +  [ b_1 - (b_1 + b_2)Y + c_2(x-y)] f_y(X,Y) + \frac\beta2Y(1-Y) f_{yy}(X,Y)\Big\}.\label{eq:ex3}
}
The result now follows from subtracting~\eq{eq:ex2} from~\eq{eq:ex3}, taking absolute values and Theorems~\ref{thm:fcts}.
\end{proof}

The proof of Theorem~\ref{thm:ex} uses Stein's method for distributional approximation, first developed for the Normal distribution in~\cite{stein72}, and which as since been developed numerous distributions and applications, see for example the surveys~\cite{Ross11, Chatterjee2014}. The theorem is in the vein of an `exchangeable pairs' approximation theorem, e.g., normal \cite[Theorem~1.1]{RinottRotar1997}; multivariate normal \cite[Theorem~2.3]{ChatterjeeMeckes2008} \cite[Theorem~2.1]{ReinertRollin2009}; 
Poisson~\cite[Proposition~3]{Chatterjeeetal2005}; 
translated Poisson~\cite[Theorem~3.1]{Rollin2007};
exponential \cite[Theorem~1.1]{Chatterjeeetal2011}, \cite[Theorem~1.1]{FulmanRoss2013}; beta \cite[Theorem~4.4]{Dobler2015}; Dirichlet \cite[Theorem~3]{GRR17}, Poisson-Dirichlet and Dirichlet processes~\cite[Theorem~1.3]{GanRoss2019}; limits in Curie-Weiss models \cite[Theorem~1.1]{ChatterjeeShao2011}. Note that following~\cite{Rollin2008}, we have relaxed the exchangeability assumption to only require $(X,Y)$ to be equal in law to $(X',Y')$.

%%%%%%%%%%%%%%%%%%%%%%%%%%%%%%%%%%%%%%
\subsection{Proof of Theorem~\ref{thm:WF}}
Let $(X,Y)$ be distributed as the proportions of type 1 for both islands in stationarity for the two-island Markov chain described in the introduction, and let $(X',Y')$ be one step ahead in Markov chain. We now define some auxiliary variables.

Let $A$ ($D$) and the number of individuals of type $1$ associated with $X'$ ($Y'$) with a parent from the first (second) island. Let $B$ ($C$) be the number of individuals of type $1$ that migrated from the first (second) island of that are of type $1$. With these definitions, $(NX',MY')|(X,Y) = (A + C, B + D)|(X,Y)$. Furthermore,
\ba{
A|(X,Y) &\sim \Bin(N-c, p_1 + (1-p_1 - p_2)X),\\
B|(X,Y) &\sim \Bin(c,p_1 + (1-p_1 - p_2)X),\\
C|(X,Y) &\sim \Bin(c, q_1 + (1-q_1-q_2)Y),\\
D|(X,Y) &\sim \Bin(M-c, q_1 + (1-q_1-q_2)Y).
}
Importantly, the above are all conditionally independent. Hence we now have
\ban{
\E[ X' -X| X,Y] &= \frac1N\E[ A + C - NX | X,Y] \\
	&=\frac{N-c}{N} (p_1 + (1-p_1 - p_2)X) + \frac{c}{N} ( q_1 + (1-q_1 - q_2)Y) - X\notag\\
	&=\frac{N-c}{N} (p_1 - (p_1 + p_2)X) + \frac{c}{N} (Y-X) + \frac{c}{N}(q_1 - (q_1 + q_2)Y).\label{eq:XX}
}
Similarly,
\ban{
\E[ Y' - Y | X,Y] &= \frac{M-c}{M} (q_1 - (q_1-q_2)Y) + \frac{c}{M}(X-Y) + \frac{c}{M}(p_1 - (p_1 + p_2)X).\label{eq:YY}
}
To apply Theorem~\ref{thm:ex}, we proceed to bound the required terms in the following pair of lemmas.
\begin{lemma}\label{lem:2mom}
Let $\lambda = \frac{1}{2N}$, $\alpha = 2$ and $\beta = \frac{2N}{M}$, then
\ban{
\E &\bbbabs{ \frac{\alpha}{2} X(1-X) - \frac{1}{2\lambda} \E[ (X'-X)^2 | X,Y] }\notag\\
	&\leq N(p_1+p_2)^2 +Np_1^2 + (4c+3 + 2Np_1)(p_1+p_2) + (4c+3)p_1+\frac{3c^2+c}{N}\notag\\
	&\hspace{0.5cm} + ((4c+2)p_1+ 2cq_1)(p_1+p_2) + \frac{1}{N} (2c^2(p_1+q_1) + c(p_1+p_2)),\label{ti2moma}\\
\E &\bbbabs{ \frac{\beta}{2} Y(1-Y) - \frac{1}{2\lambda} \E[ (Y'-Y)^2 | X,Y] } \notag\\
	&\leq \frac{N}{M} \Big\{ M(q_1+q_2)^2 +Mq_1^2 + (4c+3 + 2Mq_1)(q_1+q_2) + (4c+3)q_1+\frac{3c^2+c}{M}\notag\\
	&\hspace{0.5cm} + ((4c+2)q_1+ 2cp_1)(q_1+q_2) + \frac{1}{M} (2c^2(p_1+q_1) + c(q_1+q_2))\Big\}.\label{ti2momb}\\
\frac{1}{\lambda}\Big|\E[& (X'-X)(Y'-Y)|X,Y]\Big|\notag\\
	&\hspace{0.5cm} \leq 2N\Big( p_1 + p_2 + \frac{c}{N} (1 + q_1 + q_2) \Big) \Big( q_1 + q_2 + \frac{c}{M} (1 + p_1 + p_2) \Big).\label{ti2momc}
}
\end{lemma}
\begin{proof}
We begin with proving~\eq{ti2moma}.
\ban{
\E[ (X')^2 | X,Y] &= \frac{1}{N^2}\E[ (A+C)^2 | X,Y]\notag\\
	&=\frac{1}{N^2} \E [ A(A-1) + A + C(C-1) + C + 2AC | X,Y]\notag\\
	&= \frac{1}{N^2} \Bigg[ (N-c)(N-c-1) (p_1 + (1-p_1-p_2)X)^2 + (N-c)(p_1 + (1-p_1-p_2)X)\notag\\
	&\hspace{.5cm} + c(c-1)(q_1 + (1-q_1-q_2)Y)^2 + c(q_1 + (1-q_1-q_2)Y)\notag\\
	&\hspace{.5cm} + 2(N-c)c(p_1 + (1-p_1-p_2)X)(q_1 + (1-q_2-q_2)Y)\Bigg].\label{eq:x2mom}
}
Furthermore,
\ba{
\E [X' X | X,Y] = \frac{N-c}{N} (p_1X + (1-p_1 - p_2)X^2) + \frac{c}{N}(q_1X + (1-q_1-q_2)XY).
}
For readability, we treat each power of $X$ in $\frac{1}{2\lambda}\E[ (X'-X)^2 | X,Y]$ separately. First, the coefficient of $X^2$ is as follows.
\ban{
\frac{1}{N} & \Bigg[ (N-c)(N-c-1)(1-p_1-p_2)^2 - 2(N^2-cN)(1-p_1-p_2) + N^2 \Bigg] \notag\\
	&=\frac{1}{N} \Bigg[ N^2[ (1-p_1-p_2)^2 -2(1-p_1-p_2) + 1] \notag\\
	&\hspace{.5cm}+ N[(-2c-1)(1-p_1-p_2)^2 + 2c(1-p_1-p_2)]+ c(c-1)(1-p_1-p_2)^2\Bigg]\notag\\
	&= -1 + N (p_1+p_2)^2  + 2c [ (p_1+p_2) - (p_1+p_2)^2] + 2(p_1+p_2) - (p_1+p_2)^2 \notag\\
	&\hspace{.5cm}+ \frac{1}{N} c(c-1) (1-p_1-p_2)^2.\label{ti2X2}
}
For the coefficient of $X$,
\ban{
\frac{1}{N}& \Big\{ (N-c)(N-c-1)2p_1(1-p_1-p_2) + (N-c)(1-p_1-p_2)\notag \\
	&\hspace{.5cm}+ 2c(N-c) (1-p_1-p_2)q_1 - 2N(N-c)p_1 -2cNq_1) \Big\}\notag\\
	&= N \Big[ 2p_1(1-p_1-p_2) - 2p_1 \Big]\notag\\
	&\hspace{.5cm} +  \Big[ -2(2c+1)p_1(1-p_1-p_2) + (1-p_1-p_2) + 2cq_1(1-p_1-p_2) + 2c(p_1-q_1) \Big]\notag\\
	&\hspace{.5cm}+ \frac{1}{N} \Big[2c(c-1)p_1(1-p_1-p_2) - c(1-p_1-p_2) - 2c^2(1-p_1-p_2)q_1\Big]\notag\\
	&= 1 - 2Np_1(p_1+p_2) -2(2c+1)p_1(1-p_1-p_2) - (p_1+p_2) - 2cq_1(p_1+p_2) + 2cp_1 \notag\\
	&\hspace{.5cm} + \frac{1}{N} \Big[ (2c(c-1)p_1 - c + 2c^2q_1)(1-p_1-p_2) \Big]\notag\\ 
	&= 1 + \Big[ -2Np_1(p_1+p_2) - 2(2c+1)p_1 - (p_1+p_2) + 2cp_1 - \frac cN\Big]\notag\\ 
	&\hspace{.5cm} + \Big[ 2(2c+1)p_1(p_1+p_2) - 2cq_1(p_1+p_2) + \frac1N \big(( 2c(c-1) + 2c^2q_1)(1-p_1-p_2) + c(p_1+p_2) \big) \Big],\label{ti2X}
}
Where in the last equality we have sorted the terms into leading and smaller order terms. For the remaining terms,
\ban{
&\frac{1}{N}\Big| (N-c)(N-c-1)p_1^2 + (N-c)p_1 + c(c-1)(q_1 + (1-q_1-q_2)Y)^2 + c(q_1 + (1-q_1-q_2)Y) \notag\\
	&\hspace{0.2cm}+ 2c(N-c)p_1(q_1 + (1-q_1-q_2)Y) +[2c(N-c)(1-p_1-p_2)(1-q_1-q_2)-2cN(1-q_1-q_2)]XY\Big|\notag\\
	&\leq Np_1^2 + p_1 +\frac{1}{N}\Big| c(c-1)(q_1 + (1-q_1-q_2)Y)^2 + c(q_1 + (1-q_1-q_2)Y)\notag \\
	&+2c(N-c)p_1(q_1 + (1-q_1-q_2)Y) + [2cN(-p_1-p_2)(1-q_1-q_2)-2c^2(1-p_1-p_2)(1-q_1-q_2)]XY \Big|\notag\\
	& \leq Np_1^2 + p_1 + \frac{2c^2}{N} + 2c(p_1+p_2).\label{ti2C}
}
Taking absolute values of~\eq{ti2X2},~\eq{ti2X},~\eq{ti2C} and some elementary bounding yields,
\ba{
\E &\bbbabs{ \frac{\alpha}{2} X(1-X) - \frac{1}{2\lambda} \E[ (X'-X)^2 | X,Y] }\\
	&\leq N(p_1+p_2)^2 + (2c+2)(p_1+p_2) + \frac{c^2}{N} + (2Np_1+1)(p_1+p_2) + (4c+2)p_1+ \frac cN\\
	&+ ((4c+2)p_1+ 2cq_1)(p_1+p_2) + \frac{1}{N} (2c^2(p_1+q_1) + c(p_1+p_2)) \\
	&+  Np_1^2 + p_1 + \frac{2c^2}{N} + 2c(p_1+p_2)\\
	&\leq N(p_1+p_2)^2 +Np_1^2 + (4c+3 + 2Np_1)(p_1+p_2) + (4c+3)p_1+\frac{3c^2+c}{N}\\
	&\hspace{0.5cm} + ((4c+2)p_1+ 2cq_1)(p_1+p_2) + \frac{1}{N} (2c^2(p_1+q_1) + c(p_1+p_2)).
}
The second bound~\eq{ti2momb} follows via an analogous derivation up to a scaling constant. For the final bound~\eq{ti2momc}, first recall that $X'$ and $Y'$ are conditionally independent, then using~\eq{eq:XX} and~\eq{eq:YY},
\ba{
\bbabs{\E[ &(X'-X)(Y'-Y)|X,Y]} = \bbabs{\E[ X'-X| X,Y]\E[Y'-Y|X,Y]}\\
	&\leq \bbbabs{ \Bigg( \frac{N-c}{N}(p_1 - (p_1+p_2)X ) + \frac cN(Y-X) + \frac cN (q_1 -(q_1+q_2)Y) \Bigg)}\\
	&\hspace{1cm} \cdot  \bbbabs{ \Bigg( \frac{M-c}{M}(q_1 - (q_1+q_2)Y ) + \frac cM(X-Y) + \frac cM (p_1 -(p_1+p_2)Y) \Bigg)}\\
	&\leq \Big( p_1 + p_2 + \frac{c}{N} (1 + q_1 + q_2) \Big) \Big( q_1 + q_2 + \frac{c}{M} (1 + p_1 + p_2) \Big)
}
\end{proof}
\begin{lemma}\label{lem:3mom}
Let $\lambda = \frac{1}{2N}$, and set
\ba{
\eps_X &:= (p_1 + p_2) + \frac{c}{N}(1 + q_1 + q_2),\\
\eps_Y &:= (q_1 + q_2) + \frac{c}{M}(1+q_1+q_2).
}
Then,
\ban{
\frac{1}{\lambda} \E |(X'-X)^3| &\leq \Bigg[  \frac{1}{N} \big(  4N^2 + 4c^2 + 6cN \Big)^\frac{1}{4} +\eps_X \Bigg]^2 \Bigg[ 2\sqrt{N} + 2N \eps_X \Bigg],\label{ti3moma}\\
\frac{1}{\lambda} \E| (X'-X)^2 (Y'-Y)| &\leq \Bigg[  \frac{1}{N} \big(  4N^2 + 4c^2 + 6cN \big)^\frac{1}{4} +\eps_X \Bigg]^2 \Bigg[ \frac{2N}{\sqrt M} + 2N\eps_Y \Bigg],\label{ti3momb}\\
\frac{1}{\lambda}\E| (X'-X) (Y'-Y)^2| & \leq  \Bigg[  \frac{1}{M} \big(  4M^2 + 4c^2 + 6cM \big)^\frac{1}{4} +\eps_Y \Bigg]^2 \Bigg[ 2\sqrt N +2 N\eps_X \Bigg],\label{ti3momc}\\
\frac{1}{\lambda}\E| (Y'-Y)^3| & \leq \Bigg[  \frac{1}{M} \big(  4M^2 + 4c^2 + 6cM \big)^\frac{1}{4} +\eps_Y \Bigg]^2 \Bigg[ \frac{2N}{\sqrt M} + 2N\eps_Y \Bigg].\label{ti3momd}
}
\end{lemma}
\begin{proof}
We begin with proving~\eq{ti3moma}. Using H\"older's inequality
\ba{
\E |X'-X|^3 &\leq \Big( \E(X'-X)^4 \Big)^{\frac12} \Big( \E(X'-X)^2 \Big)^{\frac12}.
}
Using Minkowski's inequality,
\ban{
[\E(X'-X)^4]^{\frac14} &= \Big[\E\big( X' -\E(X'|X,Y) + \E(X'|X,Y) - X\big)^4\Big]^{\frac14}\notag\\
	&\leq \Big[ \E\big(X'-\E(X'|X,Y) \big)^4 \Big]^{\frac14} + \Big[ \E\big(\E(X'|X,Y) - X \big)^4 \Big]^{\frac14}.\label{ti3A}
}
We further decompose
\ba{
\E\big(X'-\E(X'|X,Y) \big)^4 &= \frac{1}{N^4} \E [ A - \E(A|X,Y) + C - \E(C|X,Y) ]^4\\
	&=\frac{1}{N^4} \E \Big[ \E\big(A - \E(A|X,Y) | X,Y\big)^4 + \E\big(C - \E(C|X,Y) | X,Y\big)^4\\
	&\ \  + 6 \E\big(A - \E(A|X,Y) | X,Y\big)^2\E\big(C - \E(C|X,Y) | X,Y\big)^2 \Big].
}
Note that for $W \sim \Bin(n,p)$,
\ban{
\E(W - np)^4 &= 3(np(1-p))^2 + np(1-p)(1-6p(1-p)) \leq 3(np(1-p))^2 + np(1-p) \leq 4n^2,\label{binom4}\\
\E(W - np)^2 &= np(1-p) \leq n\label{binom2},
}
Then recalling the definitions of $A$ and $C$ at the start of this section,
\ban{
\E\big(X'-\E(X'|X,Y) \big)^4 &\leq \frac{1}{N^4} \Big( 4(N-c)^2 + 4c^2 + 6(N-c)c \Big) \leq \frac{1}{N^4} \Big( 4(N^2 + c^2) + 6cN \Big).\label{ti3a}
}
From~\eq{eq:XX}, the following term ends up being smaller than the leading order term, so we simply denote it by $\eps_X$ for ease of reading.
\ban{
\abs{\E( X' - X|X,Y)} \leq (p_1 + p_2) + \frac{c}{N}(1 + q_1 + q_2) =: \eps_X. \label{ti3b}
}
Using~\eq{ti3a} and~\eq{ti3b} in~\eq{ti3A} yields
\ba{
[ \E(X'-X)^4]^\frac14 \leq \frac{1}{N}\Big( 4N^2 + 4c^2 + 6cN \Big)^{\frac14} + \eps_X.
}
In a similar manner,
\ba{
\big[\E(X'-X)^2\big]^\frac12 &= \big[\E( X' - \E(X'|X,Y) + \E(X'|X,Y) - X)^2\big]^\frac12\\
	&\leq [\E( X' - \E(X'|X,Y))^2]^\frac12 + [\E( \E(X'|X,Y) - X)^2]^\frac12.
}
And
\ba{
\E(X'-\E(X'|X,Y))^2 &=\frac{1}{N^2} \E\Big[  \E\big[(A - \E(A|X,Y))^2 | X,Y\big] + \E\big[(C - \E(C|X,Y))^2 | X,Y\big] \Big]\\
	&\leq\frac{1}{N^2}(N-c+c) = \frac{1}{N}.
}
Then,
\ba{
\big[\E(X'-X)^2\big]^\frac12 \leq \frac{1}{\sqrt N} + \eps_X.
}
Hence,
\ba{
\E|(X'-X)^3| \leq \Big[  \frac{1}{N} \big(  4N^2 + 4c^2 + 6cN \big)^\frac{1}{4} +\eps_X \Big]^2 \Big[ \frac{1}{\sqrt N} + \eps_X \Big]
}
For the second bound~\eq{ti3moma}, we analogously define
\ba{
\eps_Y = (q_1 + q_2) + \frac{c}{M}(1+p_1 + p_2). 
}
Then in a similar fashion to the bounds for $\E|X'-X|^3$,
\ba{
\E| (X'-X)^2 (Y'-Y)| &\leq \Big[\E (X'-X)^4 \Big]^{\frac12} \Big[\E (Y'-Y)^2 \Big]^{\frac12}\\
	&\leq  \Bigg[  \frac{1}{N} \big(  4N^2 + 4c^2 + 6cN \big)^\frac{1}{4} +\eps_X \Bigg]^2 \Bigg[ \frac{1}{\sqrt M} + \eps_Y \Bigg]
}
The final two bounds~\eq{ti3momc} and~\eq{ti3momd} follow from symmetric arguments.
\end{proof}
\begin{proof}[Proof of Theorem~\ref{thm:WF}]
Given~\eq{eq:XX} and~\eq{eq:YY}, we apply Theorem~\ref{thm:ex} by setting $\l = \frac{1}{2N}$, $a_1 = 2p_1(N-c)$, $a_2 = 2p_2(N-c)$, $b_1 = 2N\frac{M-c}{M}q_1$, $b_2 = 2N \frac{M-c}{M} q_2$, $c_1 = 2c$, $c_2 = \frac{2cN}{M}$, $R_1 = \frac{c}{N}(q_1 - (q_1 + q_2)Y)$ and $R_2 =  \frac{c}{M}(p_1 - (p_1 + p_2)X)$. The theorem now follows from Lemmas~\ref{lem:2mom} and~\ref{lem:3mom} and routine bounds for $R_1$ and $R_2$.
\end{proof}
\subsection{Proof of Theorem~\ref{thm:SB}}
We will again prove the above bounds using Theorem~\ref{thm:ex}. As before, let $(X,Y)$ be a random vector in stationarity for the seed-bank Markov chain described in the introduction, and let $(X',Y')$ be one step ahead in Markov chain. We now again define some auxiliary variables.

Let $A$ be the number of offspring of type $1$ for the active population. Let $B$ ($C$) be the number of individuals of type $1$ that migrated from the active (seed-bank) population to the seed-bank (active population). Finally let $D$ be the number of individuals of type $1$ that did not migrate from the seed-bank. With these definitions, $(X',Y')|(X,Y) = (A + C, B + D)|(X,Y)$. Furthermore,
\ba{
A|(X,Y) &\sim \Bin(N-c, p_1 + (1-p_1 - p_2)X),\\
B|(X,Y) &\sim \Bin(c,p_1 + (1-p_1 - p_2)X),\\
C|(X,Y) &\sim \Hg(M,MY,c),\\
D|(X,Y) &\sim \Hg(M,MY,M-c).
}
Importantly, the above random variables are pairwise conditionally independent except $D|(X,Y) = M-C|(X,Y)$.
\ban{
\E[ X'-X | X,Y] &= \frac{N-c}{N} (p_1 + (1-p_1-p_2)X) + \frac{c}{N} Y - X\notag\\
	&= \frac{N-c}{N}(p_1 - (p_1 + p_2)X) + \frac{c}{N} (Y-X),\label{eq:XX2}\\
\E[Y'-Y | X,Y] &= \frac{M-c}{M} Y + \frac{c}{M} (p_1 + (1-p_1-p_2)X) - Y\notag\\
	&= \frac{c}{M}(X-Y) + \frac{c}{M}(p_1 - (p_1 + p_2)X).\label{eq:YY2}
}
We again apply Theorem~\ref{thm:ex} and begin with the following pair of lemmas.
\begin{lemma}\label{lem:SB2}
Let $\lambda = \frac{1}{2N}$, $\alpha = 2$ and $\beta = 0$, then
\ban{
\E &\bbbabs{ \frac{\alpha}{2} X(1-X) - \frac{1}{2\lambda} \E[ (X'-X)^2 | X,Y] }\notag\\
	&\leq N (p_1+p_2)^2 + Np_1^2+(4c+3 + (4c+2+2N)p_1)(p_1+p_2) + (2c+3)p_1\notag\\
	&\hspace{1cm} +  \frac{1}{N}(3c^2 + 2c^2p_1 + cp_2).\label{sbmom2X}\\
\E &\bbbabs{ \frac{\beta}{2} Y(1-Y) - \frac{1}{2\lambda} \E[ (Y'-Y)^2 | X,Y] }\notag\\
	& \leq \frac{4c^2N(1 + p_1)}{M(M-1)} \label{sbmom2Y},\\
\frac{1}{\lambda} \big|\E& [ (X'-X)(Y'-Y) | X,Y]\big|\notag \\
	&\leq\frac{2}{M} \Big[ (2cN + cNp_1)(p_1+p_2) + (2c^2 + 2cN)p_1 + 3c^2+cNp_1^2+ \frac{c^2M}{M-1}\Big].\label{sbmom2XY}
}
\end{lemma}
\begin{proof}
We first prove the bound~\ref{sbmom2X}.
\ban{
\E[ (X')^2 | X,Y] &= \frac{1}{N^2} \E[ (A+C)^2 | X,Y]\notag\\
	&= \frac{1}{N^2} \E[ A(A-1) + A + C(C-1) + C + 2AC | X,Y]\notag\\
	&= \frac{1}{N^2} \Big[ (N-c)(N-c-1) (p_1 + (1-p_1-p_2)X)^2 + (N-c)(p_1 + (1-p_1-p_2)X)\notag\\
	&\hspace{1cm} + \frac{MY(MY-1)c(c-1)}{M(M-1)} + cY + 2cY(N-c)(p_1 + (1-p_1-p_2)X) \Big]\label{sbmomx2},
}
and
\ba{
E[X'X|X,Y] = \frac{(N-c)}{N}(p_1X + (1-p_1-p_2)X^2) + \frac{c}{N} XY.
}
As before, for $\frac{1}{2\lambda} \E[ (X'-X)^2|X,Y]$, we deal with the coefficients of different powers of $X$ separately, starting with $X^2$.
\ban{
\frac{1}{N} &\Big[ (N-c)(N-c-1)(1-p_1-p_2)^2 - 2N(N-c)(1-p_1-p_2) + N^2\Big]\notag\\
	&= \frac{1}{N} \Big[ N^2[ (1-p_1-p_2)^2 - 2(1-p_1-p_2) + 1] \notag\\
	&\hspace{1cm}+ N[ (-2c-1)(1-p_1-p_2)^2 + 2c(1-p_1-p_2)]\notag\\
	&\hspace{1cm} + c(c-1)(1-p_1-p_2)^2\Big]\notag\\
	&= -1 + N (p_1+p_2)^2  + 2c [ (p_1+p_2) - (p_1+p_2)^2] + 2(p_1+p_2) - (p_1+p_2)^2\notag \\
	&\hspace{1cm}+ \frac{1}{N} c(c-1) (1-p_1-p_2)^2.\label{sbx2a}
}
For the coefficient of $X$,
\ban{
\frac{1}{N}& \Big[ (N-c)(N-c-1)2p_1(1-p_1-p_2) + (N-c)(1-p_1-p_2) - 2N(N-c)p_1\Big]\notag\\
	&= N\big[ 2p_1(1-p_1-p_2) - 2p_1)\big]\notag\\
	&\hspace{1cm} + \big[ (-2c-1)2p_1(1-p_1-p_2) + (1-p_1-p_2) + 2cp_1 \big]\notag\\
	&\hspace{1cm} + \frac1N \big[c(c-1)2p_1(1-p_1-p_2)-c(1-p_1-p_2)\big]\notag\\
	&= 1 -2Np_1(p_1+p_2) -(2c+2)p_1 - (p_1+p_2)\notag\\
	&\hspace{1cm} + (4c+2)p_1(p_1+p_2) + \frac{1}{N} ( 2c(c-1)p_1(1-p_1-p_2) + c(p_1+p_2))\big]\label{sbx2b}.
}
And finally the remaining terms,
\ban{
\frac{1}{N} &\Big| (N-c)(N-c-1)p_1^2 + (N-c)p_1 + \frac{MY(MY-1)c(c-1)}{M(M-1)} \notag\\
	&\hspace{1cm}+ (c + 2c(N-c)p_1)Y + (2c(N-c)(1-p_1-p_2)-2cN)XY \Big|\notag\\
	&\leq \frac{1}{N}\Big|N^2p_1^2 + Np_1 + c(c-1) + c + 2c(N-c)(p_1Y - (p_1 + p_2)XY) - 2c^2 (1-p_1-p_2)XY \Big|\notag \\ 
	&\leq Np_1^2+p_1 + \frac{2c^2}{N} + 2c(p_1 +p_2).\label{sbx2c}
}
Combining~\eq{sbx2a},~\eq{sbx2b},~\eq{sbx2c},
\ba{
\E &\bbbabs{ \frac{\alpha}{2} X(1-X) - \frac{1}{2\lambda} \E[ (X'-X)^2 | X,Y] }\\
	&\leq N (p_1+p_2)^2 +(4c+3 + 2Np_1)(p_1+p_2) + (2c+3)p_1 + Np_1^2 + \frac{3c^2}{N}\\
	&\hspace{1cm} + (4c+2)p_1(p_1+p_2) + \frac{1}{N}(2c^2p_1 + cp_2).
}
For the second bound~\eq{sbmom2Y}, 
\ban{
\E[ &Y'^2| X,Y] = \frac{1}{M^2} \E[ (B+D)^2 | X,Y]\notag\\
	&=\frac{1}{M^2} \E[ B(B-1) + B + D(D-1) +D + 2BD|X,Y]\notag\\
	&= \frac{1}{M^2} \Big[ \frac{MY(MY-1)(M-c)(M-c-1)}{M(M-1)} + (M-c)Y + c(c-1)(p_1 + (1-p_1-p_2)X)^2\notag\\
	&\hspace{1cm} + c(p_1 + (1-p_1-p_2)X) + 2(M-c)Yc(p_1 + (1-p_1-p_2)X)\Big]\label{sbmomy2}.
}
And
\ba{
\E[Y'Y | X,Y] = \frac{M-c}{M}\ Y^2 + \frac{c}{M}(p_1Y + (1-p_1-p_2)XY).
}
Putting the above together yields
\ba{
\E[ &(Y'-Y)^2|X,Y]\\ 
	&= \frac{c(c-1)}{M(M-1)} Y^2 + \Big[ \frac{c(M-c)}{M^2(M-1)} -\frac{2c^2p_1}{M^2} \Big] Y +  \frac{1}{M^2} \Big[c(c-1)(p_1 + (1-p_1-p_2)X)^2\\
	& + c(p_1 + (1-p_1-p_2)X) +[2c(M-c)(1-p_1-p_2)-2cM(1-p_1-p_2)]XY\Big].
}
The bound~\eq{sbmom2Y} now follows from taking absolute values and elementary steps. For the final bound~\eq{sbmom2XY},
\ba{
\E[ (X'-X)(Y'-Y) | X,Y] = \frac{1}{NM} \E[ (A+C-NX)(B+D-MY) | X,Y].
}
The complication in this computation is that $A,B,C,D$ are all conditionally pairwise independent except for $C$ and $D$. Therefore,
\ban{
\E[ (X'-X)&(Y'-Y) | X,Y] = \frac{1}{NM} \Big[ \E(A - NX | X,Y)\E(B+D-MY|X,Y)\notag \\
	&+ \E(CD|XY) + \E(C|X,Y)\E(B-MY|X,Y) \Big].\label{sbmomxy}
}
Now noting $C|(X,Y) = (MY-D) | (X,Y)$,
\ba{
\E[ CD | X,Y] &= \E[ D(MY-D)|X,Y] = MY\E[D|X,Y] - \E[D(D-1) + D|X,Y]\\
	&= MY(M-c)Y - \frac{MY(MY-1)(M-c)(M-c-1)}{M(M-1)} - (M-c)Y.
}
Hence,
\ba{
\frac{1}{\lambda} \big|\E[ &(X'-X)(Y'-Y) | X,Y]\big| \\
	&= \frac{2}{M} \Big| [ (N-c)(p_1 + (1-p_1-p_2)X) - NX][(M-c)Y + c(p_1+(1-p_1-p_2)X) - MY]\\
	&\hspace{1cm}+ MY(M-c)Y - \frac{MY(MY-1)(M-c)(M-c-1)}{M(M-1)} - (M-c)Y\\
	&\hspace{1cm}+cY[ c(p_1 + (1-p_1-p_2)X) - MY]\Big|\\
	&= \frac{2}{M} \Bigg| \Big[ (-N(p_1+p_2) - c(1-p_1-p_2))( c(1-p_1-p_2))\Big] X^2\\
	&\hspace{1cm} + \Big[ (-N(p_1+p_2) - c(1-p_1-p_2)) cp_1 + c(N-c)p_1(1-p_1-p_2)\Big] X\\
	&\hspace{1cm} + \Big[M(M-c) - \frac{M^2(M-c)(M-c-1)}{M(M-1)} - cM\Big] Y^2\\
	&\hspace{1cm} + \Big[ -c(N-c)p_1 + \frac{M(M-c)(M-c-1)}{M(M-1)} - (M-c) + c^2p_1 \Big] Y\\
	&\hspace{1cm} + \Big[ cN(p_1+p_2) + 2c^2(1-p_1-p_2)\Big]XY + c(N-c)p_1^2\Bigg|\\
	&\leq \frac{2}{M} \Big[ cN(p_1+p_2) + c^2 + cNp_1(p_1+p_2) + c^2p_1+cNp_1+\frac{c(c-1)M}{M-1}\\
	&\hspace{1cm} + \frac{cM}{M-1} + cNp_1 + c^2p_1 + cN(p_1+p_2) + 2c^2 + cNp_1^2 \Big]\\
	&\leq \frac{2}{M} \Big[ (2cN + cNp_1)(p_1+p_2) + (2c^2 + 2cN)p_1 + 3c^2+ cNp_1^2+\frac{c^2M}{M-1}\Big].
}
\end{proof}
\begin{lemma}\label{lem:SB3}
Let $\lambda = \frac{1}{2N}$ and set
\ba{
\eps_{M,c} := \frac{cM}{4(M-1)(M-2)(M-3)} \Big[ M^2(1+c) + M(1 + 6c + c^2) + 6c^2\Big].
}
Then,
\ban{
\frac{1}{\lambda}&\E |(X'-X)^3|\notag \\
	&\leq  \Bigg[\frac1N\Big(4N^2 + 2cN + \eps_{M,c} \Big)^\frac14 + p_1+p_2+\frac{c}{N}\Bigg]^2\Bigg[ \sqrt N + N(p_1+p_2)+c \Bigg],\label{sb3moma}\\
\frac{1}{\lambda}&\E| (X'-X)^2 (Y'-Y)| \notag\\
	&\leq \Bigg[\frac1N\Big( 4N^2 +2cN + \eps_{M,c}  \Big)^\frac14 + p_1+p_2+\frac{c}{N}\Bigg]^2 \Bigg[\sqrt{ \frac{5cN^2}{4M^2}} + \frac{cN}{M}(1 + p_1 + p_2)\Bigg],\label{sb3momb}\\
\frac{1}{\lambda}&\E| (X'-X) (Y'-Y)^2|\notag \\
	& \leq   \Bigg[ \frac{1}{M} \Big(6c^2 + \eps_{M,c}\Big)^\frac14 + \frac{c}{M}(1 + p_1 + p_2) \Bigg]^2 \Bigg[\sqrt N + N(p_1+p_2)+c \Bigg],\label{sb3momc}\\
\frac{1}{\lambda}&\E| (Y'-Y)^3|\notag\\
	 & \leq \Bigg[ \frac{1}{M} \Big(6c^2 + \eps_{M,c}\Big)^\frac14 + \frac{c}{M}(1 + p_1 + p_2) \Bigg]^2 \Bigg[\sqrt{ \frac{5cN^2}{4M^2}} + \frac{cN}{M}(1 + p_1 + p_2))\Bigg].\label{sb3momd}
}
\end{lemma}
\begin{proof}
We follow a similar approach to Lemma~\ref{lem:3mom} and begin with~\eq{sb3moma}. Using H\"older's inequality,
\ban{
\E |(X'-X)^3| &\leq \Big( \E(X'-X)^4 \Big)^{\frac12} \Big( \E(X'-X)^2 \Big)^{\frac12}.\label{sb1}
}
Using Minkowski's inequality,
\ban{
[\E(X'-X)^4]^{\frac14} &= \Big[\E\big( X' -\E(X'|X,Y) + \E(X'|X,Y) - X\big)^4\Big]^{\frac14}\notag\\
	&\leq \Big[ \E\big(X'-\E(X'|X,Y) \big)^4 \Big]^{\frac14} + \Big[ \E\big(\E(X'|X,Y) - X \big)^4 \Big]^{\frac14}\label{sbAAA}.
}
We further decompose
\ban{
\E\big(X'-\E(X'|X,Y) \big)^4 &= \frac{1}{N^4} \E [ A - \E(A|X,Y) + C - \E(C|X,Y) ]^4\notag\\
	&=\frac{1}{N^4} \E \Big[ \E\big(A - \E(A|X,Y) | X,Y\big)^4 + \E\big(C - \E(C|X,Y) | X,Y\big)^4\notag\\
	&\ \  + 6 \E\big(A - \E(A|X,Y) | X,Y\big)^2\E\big(C - \E(C|X,Y) | X,Y\big)^2 \Big].\label{sbBIG}
}
In addition to the binomial moments~\eq{binom4} and~\eq{binom2} used earlier, we also require Hypergeometric moments. For $W \sim \Hg(N,D,n)$
\ban{
\E\left(W - \frac{nD}{N}\right)^4 &= \frac{nD(N-D)(N-n)}{N^4(N-1)(N-2)(N-3)} \notag\\
	&\hspace{1cm} \cdot\Big[ N^2(6n^2 + N - 6nN + N^2) + 3D(D-N) (2N^2 + (n^2 - nN)(6+N)) \Big],\label{hgmom4}\\
\E\left(W - \frac{nD}{N} \right)^2 &= \frac{nD(N-D)(N-n)}{N^2(N-1)}.\label{hgmom2}
}
Then using~\eq{hgmom4}
\ban{
\E[ (C - &E(C|X,Y))^4 | X,Y]\notag\\
	& = \frac{cMY(M-MY)(M-c)}{M^4(M-1)(M-2)(M-3)}\Big[ M^2(6c^2 + M - 6cM + M^2)\notag \\
	&\ \ \ + 3MY(MY-M)(2M^2 + ( c^2 - cM)(6+M))\Big]\notag\\
	&= \frac{cMY(M-MY)(M-c)}{M^4(M-1)(M-2)(M-3)}\Big[ M^4 (1 - 6 Y + 3 c Y + 6 Y^2 - 3 c Y^2) \notag\\
	&\ \ \ + M^3 (1 - 6 c + 18 c Y - 3 c^2 Y - 18 c Y^2 + 3 c^2 Y^2) +  M^2 (6 c^2 - 18 c^2 Y + 18 c^2 Y^2)\Big]\notag\\
	&\leq \frac{cM}{4(M-1)(M-2)(M-3)} \Big[ M^2(1+c) + M(1 + 6c + c^2) + 6c^2\Big] =: \eps_{M,c}.\label{sbA1}
}
From~\eq{hgmom2},
\ban{
\E( (C - &E(C|X,Y))^2 | X,Y) = \frac{cMY(M-MY)(M-c)}{M^2(M-1)} \leq \frac{c}{4}.\label{sbA2}
}
Furthermore from~\eq{eq:XX2}, 
\ban{
\abs{ \E(X'|X,Y) - X} \leq p_1 + p_2 + \frac{c}{N}.\label{sbA2a}
}
Using~\eq{binom2} and~\eq{binom4}, $\E (A - E(A|X,Y))^4 \leq 4(N-c)^2$, $\E (A - E(A|X,Y))^2 \leq N-c$. Then combining~\eq{sbA1},~\eq{sbA2} into~\eq{sbBIG} and~\eq{sbA2a} into~\eq{sbAAA} yields
\ban{
[\E(X'-X)^4]^\frac14 \leq \frac1N\Bigg[4N^2 + 2cN + \eps_{M,c} \Bigg]^\frac14 + p_1+p_2+\frac{c}{N}. \label{sbAA}
}
For the second moment in~\eq{sb1}, 
\ban{
[\E(X'-X)^2]^\frac12 &= [\E(X'-\E(X'|X,Y) + \E(X'|X,Y) - X)^2]^\frac12\notag\\
	&=[ \E(X' - \E(X'|X,Y))^2]^\frac12 + [\E(\E(X'|X,Y) - X)^2]^\frac12,\label{sbA4}
}
and using~\eq{binom2} and~\eq{sbA2}
\ban{
\E(X'-\E(X'|X,Y))^2 &=\frac{1}{N^2} \E\Big[  \E\big(A - \E(A|X,Y) | X,Y\big)^2 + \E\big(C - \E(C|X,Y) | X,Y\big)^2 \Big]\notag\\
	&\leq\frac{1}{N^2}\left(N-c+\frac c4\right) = \frac{1}{N}.\label{sbA5}
}
Then from~\eq{sbA5} and~\eq{sbA2a},
\ban{
[\E(X'-X)^2]^\frac12 \leq \frac{1}{\sqrt N} + p_1+p_2+\frac cN.\label{sbAB}
}
Combining~\eq{sbAA} and~\eq{sbAB} in~\eq{sbAAA} yields~\eq{sb3moma}. 	
\ba{
\E |(X'-X)^3| \leq  \Bigg[\frac1N\bigg(4N^2 +2cN + \eps_{M,c} \bigg)^\frac14 + p_1+p_2+\frac{c}{N}\Bigg]^2\Bigg[ \frac{1}{\sqrt N} + p_1+p_2+\frac{c}{N} \Bigg].
}
For~\eq{sb3momb},
\ban{
\E \abs{(X'-X)^2 (Y'-Y)} \leq \big[ \E(X'-X)^4\big]^\frac{1}{2} \big[\E(Y'-Y)^2\big]^\frac12.\label{sbBB}
}
For the second half of the above,
\ban{
[E(Y'-Y)^2]^\frac12 &= [ \E(Y'-\E(Y'|X,Y) + \E(Y'|X,Y) - Y)^2]^\frac12\notag\\
	&= [\E(Y' - \E(Y'|X,Y))^2]^\frac12 + [\E(\E(Y'|X,Y) - Y)^2]^\frac12.\label{sbB1}
}
Furthermore noting that since $D|(X,Y) = (MY - C) | (X,Y)$, implies that $\E(D - \E(D|X,Y))^2 = \E(C - \E(C|X,Y))^2$, then using~\eq{binom2} and~\eq{sbA2},
\ba{
\E( Y'-\E(Y'|X,Y))^2 &= \frac{1}{M^2} \E\big[ (B-\E(B|X,Y))^2 + (D - \E(D|X,Y))^2\big]\\
	&\leq \frac{1}{M^2} \left(c + \frac{c}{4}\right) = \frac{5c}{4M^2}.
}
And finally from~\eq{eq:YY2}, $\abs{Y - \E(Y'|X,Y)} \leq \frac{c}{M}(1 + p_1 + p_2)$. Hence,
\ban{
[E(Y'-Y)^2]^\frac12 \leq \sqrt{ \frac{5c}{4M^2}} + \frac{c}{M}(1 + p_1 + p_2)\label{sbB2}.
}
The bound~\eq{sb3momb} now follows from using~\eq{sbAA} and~\eq{sbB2} in~\eq{sbBB}.

For the third bound~\eq{sb3momc} we begin with
\ban{
\E \abs{(Y'-Y)^2(X'-X)} \leq [ \E(Y'-Y)^4]^\frac12 [\E(X'-X)^2]^\frac12,\label{sbCC}
}
and
\ba{
[\E(Y'-Y)^4]^\frac14 &= [\E( Y' - \E(Y'| X,Y) + \E(Y'|X,Y) - Y)^4]^\frac14\\
	&\leq [\E(Y'-\E(Y'|X,Y))^4]^\frac14 + [ \E( \E(Y'|X,Y) - Y)^4 ] ^\frac14.
}
Using the fact that $\E( (C - E(C|X,Y))^4 | X,Y) = \E( (D - E(D|X,Y))^4 | X,Y)$ as $C | (X,Y) = (MY - D)|(X,Y)$ and similarly for the second moment,~\eq{sbA1},~\eq{binom4},~\eq{binom2} and~\eq{sbA2},
\ba{
\E(Y' -\E(Y'|X,Y))^4 &= \frac{1}{M^4} \E[ (B - E(B|X,Y))^4 + (D - E(D|X,Y))^4 \\
	&\hspace{1cm}+ 6 (B - E(B|X,Y))^2(D - E(D|X,Y))^2]\\
	&\leq \frac{1}{M^4} \left(4c^2 + \eps_{M,c} + \frac32c^2\right).
}
Hence,
\ban{
[\E(Y'-Y)^4]^\frac14 \leq  \frac{1}{M} [6c^2 + \eps_{M,c}]^\frac14 + \frac{c}{M}(1 + p_1 + p_2).\label{sbC1}
}
Using~\eq{sbC1} and~\eq{sbAB} in~\eq{sbCC} yields the bound.

The final bound~\eq{sb3momd} follows from ~\eq{sbC1} and~\eq{sbB2}.
\end{proof}
\begin{proof}[Proof of Theorem~\ref{thm:SB}]
Given~\eq{eq:XX2} and~\eq{eq:YY2} we apply Theorem~\ref{thm:ex}, setting $\l = \frac{1}{2N}$, $a_1 = 2p_1(N-c)$, $a_2 = 2p_2(N-c)$, $b_1 = b_2=0$, $c_1 = 2c$, $c_2 = \frac{2cN}{M}$, $\alpha = 2$, $\beta = 0$, $R_1 =0$ and $R_2 = \frac{c}{M}(p_1 - (p_1 + p_2)X)$. The proof now follows from Lemmas~\ref{lem:SB2} and~\ref{lem:SB3} and routine bounds for $R_1$ and $R_2$.
\end{proof}
\subsection{Proof of Theorem~\ref{thm:Beta}}
Our goal is to use Theorem~3 from~\cite{GRR17} to bound the difference between $\frac{N}{N+M}X + \frac{M}{N+M}Y$ and a beta random variable. The following is the linearity condition required to use~\cite[Theorem~3]{GRR17}. Recall~\eq{eq:XX} and~\eq{eq:YY},
\ban{
\frac{1}{N+M}\E\big[ &(NX' + MY') - (NX + MY) | NX + MY \big]\notag \\
	&= \frac{1}{N+M} \E \big[ Np_1 + Mq_2 - (N(p_1 + p_2)X + M(q_1 + q_2)Y) \big| NX + MY \big].\notag\\
	&= \frac{1}{2(N + M)}\Bigg[ 2(Np_1 + Mq_1) -2\big( N(p_1 + p_2) + M(q_1 + q_2)\big) \Big( \frac{N}{N+M} X + \frac{M}{N+M} Y \Big) \Bigg]\notag\\
	&\ \ \ + \frac{NM}{(N+M)^2} \big( (q_1 + q_2) - (p_1 + p_2) \big) \E \big[ X - Y | NX + MY \big].\label{eq:lin}
}
As in the previous section, we bound the necessary terms in a sequence of lemmas and we set $\lambda = \frac{1}{2(N+M)}$.
\begin{lemma}\label{lem:d2mom}
\ba{
\E \Bigg| \Big[ &\Big(\frac{N}{N+M} X + \frac{M}{N+M}Y \Big)\Big( 1- \Big( \frac{N}{N+M} X + \frac{M}{N+M}Y\Big) \Big)\Big] \\
	&- \frac{1}{2\lambda} \E \Big[ \Big(\Big( \frac{N}{N+M} X' + \frac{M}{N+M}Y' \Big) - \Big( \frac{N}{N+M} X + \frac{M}{N+M}Y \Big)\Big)^2 \Big| NX + MY \Big] \Bigg|\\
	&\leq \frac{1}{N+M} \Big[\big( N(2p_1 + p_2) + M(2q_1 + q_2) \big)^2 + 3N(2p_1 + p_2) + 3M(2q_1 + q_2)\Big]\\
	&\ \ \ + \frac{NM}{(N+M)^2} \E(X-Y)^2
}
\end{lemma}
\begin{proof}
Noting that $NX' + MY' | X,Y \ed (F + G)|X,Y$ where $F|X,Y \sim \Bin(N, p_1 + (1-p_1-p_2)X)$  and $G|X,Y \sim \Bin(M, q_1 + (1-q_1-q_2)Y)$ and $F$ and $G$ are conditionally independent,
\ba{ \frac{1}{2\lambda}& \E \Big[ \Big(\Big( \frac{N}{N+M} X' + \frac{N}{N+M}Y' \Big) - \Big( \frac{N}{N+M} X + \frac{N}{N+M}Y \Big)\Big)^2 \Big| NX + MY \Big]\\
&= \frac{1}{2\lambda (N+M)^2} \E \Bigg[\E \Big[ \big( (F + G) - ( NX + MY )\big)^2 \Big| X,Y \Big] \Bigg| NX + MY \Bigg]\\
	&=  \frac{1}{N+M} \E \Big[ \E\big[F(F-1) + F + 2FG + G(G-1) + G \\
	&\ \ \ +2(F+G)(NX + MY) + (NX + MY)^2  | X,Y\big] \Big| NX + MY \Big]\\
	&= \frac{1}{N+M} \E \Big[ N(N-1)(p_1 + (1-p_1-p_2)X)^2 + N(p_1 + (1-p_1-p_2)X) \\
	&\ \ \ + M(M-1)(q_1 + (1-q_1-q_2)Y)^2 + M(q_1 + (1-q_1 - q_2)Y) \\
	&\ \ \ + 2NM(p_1 + (1-p_1-p_2)X)(q_1 + (1-q_1-q_2)Y)\\
	&\ \ \ -2(N(p_1 + (1-p_1-p_2)X)+ M(q_1 + (1-q_1 - q_2)Y))(NX + MY)\\
	&\ \ \  + (NX + MY)^2 \Big| NX + MY\Big]\\
	&=\frac{1}{N+M} \E \Big[ X^2 \big[ N(N-1)(1-p_1-p_2)^2  - 2N^2(1-p_1-p_2) + N^2 \big]\\
	& \ \ \ + Y^2\big[ M(M-1)(1-q_1-q_2)^2 - 2M^2(1-q_1-q_2) + M^2 \big]\\
	&\ \ \ +XY\big[2NM(1-p_1-p_2)(1-q_1-q_2)-2NM(1-p_1-p_2)\\
	&\ \ \ \ \ \ - 2NM(1-q_1-q_2) + 2NM\big]\\
	& \ \ \ + X\big[ 2N(N-1)p_1(1-p_1-p_2) + N(1-p_1-p_2)\\
	&\ \ \ \ \ \  + 2NMq_1(1-p_1-p_2)- 2N(Np_1 + Mq_1)\big]\\
	&\ \ \ + Y\big[ 2M(M-1)q_1(1-q_1-q_2) + M(1-q_1-q_2) \\
	&\ \ \ \ \ \ + 2NMp_1(1-q_1-q_2) - 2M(Mq_1 + Np_1) \big]\\
	&\ \ \ + N(N-1)p_1^2 + Np_1 + M(M-1)q_1^2 + Mq_1 + 2N Mp_1q_1\Big| NX + MY \Big]\\
	&= \frac{1}{N+M}\E  \Bigg[ X^2 \Big[-N + \frac{NM}{N+M} + N^2(p_1 + p_2)^2  + 2N(p_1+p_2) - N(p_1+p_2)^2 \Big]\\
	&\ \ \ + Y^2 \Big[-M + \frac{NM}{N+M} +  M^2(q_1 + q_2)^2 + 2M(q_1 + q_2) - M(q_1 + q_2)^2\Big]\\
	&\ \ \ + XY \Big[-\frac{2NM}{N+M} + 2 NM(p_1 + p_2)(q_1 + q_2)\Big]\\
	&\ \ \ +  X \Big[ N - 2N^2p_1(p_1+p_2) - N(2p_1(1-p_1-p_2) + (p_1+p_2)) - 2NMq_1(p_1+p_2) \Big]\\
	&\ \ \ + Y\Big[ M - 2M^2q_1(q_1+q_2) - M(2q_1(1-q_1-q_2) + (q_1+q_2)) - 2NMp_1(q_1+q_2) \Big]\\
	&\ \ \ + (Np_1 + Mq_1)^2 +Np_1(1-p_1) + Mq_1(1-q_1)\\
	&\ \ \ - \frac{NM}{(N+M)} (X-Y)^2 \Bigg| NX + MY \Bigg]\\
	&= \Big(\frac{N}{N+M} X + \frac{M}{N+M}Y \Big)\Big( 1- \Big( \frac{N}{N+M} X + \frac{M}{N+M}Y\Big) \Big)\\
	&\ \ \ + \frac{1}{N+M}\E  \Bigg[ X^2 \Big[N^2(p_1 + p_2)^2  + 2N(p_1+p_2) - N(p_1+p_2)^2 \Big]\\
	&\ \ \ + Y^2 \Big[ M^2(q_1 + q_2)^2 + 2M(q_1 + q_2) - M(q_1 + q_2)^2\Big]+ XY \Big[2 NM(p_1 + p_2)(q_1 + q_2)\Big]\\
	&\ \ \ +  X \Big[- 2N^2p_1(p_1+p_2) - N(2p_1(1-p_1-p_2) + (p_1+p_2)) - 2NMq_1(p_1+p_2) \Big]\\
	&\ \ \ + Y\Big[- 2M^2q_1(q_1+q_2) - M(2q_1(1-q_1-q_2) + (q_1+q_2)) - 2NMp_1(q_1+q_2) \Big]\\
	&\ \ \ + (Np_1 + Mq_1)^2 +Np_1(1-p_1) + Mq_1(1-q_1)- \frac{NM}{(N+M)} (X-Y)^2 \Bigg| NX + MY \Bigg].
}
The result now follows from elementary bounding.
\end{proof}
\begin{lemma}\label{lem:d3mom}
\ba{
\frac{1}{\lambda} &\E \Bigg| \Big[\Big(\frac{N}{N+M}X' + \frac{M}{N+M}Y'\Big) - \Big(\frac{N}{N+M}X + \frac{M}{N+M}Y\Big) \Big]^3 \Bigg|\\
	& \leq \frac{1}{(N+M)^2} \big[ (4N^2 + 6NM + 4M^2)^{\frac 14} + p_1 + p_2 + q_1 + q_2 \big]^2 \big[\sqrt{N+M} + p_1 + p_2 + q_1 + q_2\big].
}
\end{lemma}
\begin{proof}
Using H\"older's inequality,
\ban{
\E\Big|&\big[ (NX' + MY') - (NX + MY)\big]^3\Big|\notag \\
	&\leq \Big[\E\big[ (NX' + MY') - (NX + MY)\big]^4\Big]^{\frac 12} \Big[\E\big[ (NX' + MY') - (NX + MY)\big]^2\Big]^{\frac 12}\label{eq:d3mom}
}
Then using Minkowski's inequality,
\ban{
\Big[ \E\big(& (NX' + MY') - (NX + MY)\big)^4\Big]^{\frac 14}\notag \\
	&\leq \Big[ \E\big( (NX' + MY') - \E(NX' + MY' | X,Y) \big)^4 \Big]^{\frac 14}\notag \\
	&\ \ \ + \Big[ \E\big( \E(NX' + MY' | X,Y) - (NX + MY) \big)^4 \Big]^{\frac 14}.\label{eq:d3moma}
}
Recalling the definitions of the random variables $F$ and $G$ from the previous lemma, and~\eq{binom4},~\eq{binom2},
\ban{
 \E\big( &(NX' + MY') - \E(NX' + MY' | X,Y) \big)^4\notag\\
	&= \E[ \E( (F+G) - \E(F + G))^4]\notag\\
	&= \E[ \E(F - \E(F))^4 + 6 \E(F - \E(F))^2 \E(G - \E(G))^2 + \E(G - \E(G))^4]\notag \\
	&\leq 4N^2 + 6NM + 4M^2.\label{eq:d3momb}
}
Furthermore,
\ban{
\big| E(NX' +& MY' | X,Y) - (NX + MY) \big| = | p_1 - (p_1 + p_2)X + q_1 - (q_1 + q_2)Y|\notag\\
	& \leq  p_1 + p_2 + q_1 + q_2.\label{eq:d3momc}
}
Using~\eq{binom2}, in the final inequality
\ban{
\Big[\E[ &(NX' + MY') - (NX + MY)]^2\Big]^{\frac 12}\notag\\
	&= \E\big[ \E( (NX' + MY') - \E(NX' + MY' | X,Y))^2 \big]^{\frac 12} \notag\\
	&\ \ \ + \E\big[ \E(NX' + MY' | X,Y) - (NX + MY))^2 \big]^{\frac 12}\notag\\
	&\leq \big[\E\big( \E[(F - \E(F))^2|X,Y] + \E[(G - \E(G))^2|X,Y]\big)\big]^{\frac 12}+ p_1 + p_2 + q_1 + q_2\notag\\
	&\leq (N + M)^{\frac 12} + p_1 + p_2 + q_1 + q_2.\label{eq:d3momd}
}
The result now follows from substituting~\eq{eq:d3moma}, \eq{eq:d3momb}, \eq{eq:d3momc} and \eq{eq:d3momd} into \eq{eq:d3mom}.
\end{proof}
\begin{proof}[Proof of Theorem~\ref{thm:Beta}]
We apply~\cite[Theorem 3]{GRR17}, set $a_1 = 2(Np_1 + M q_1), a_2 = 2(Np_2 + Mq_2)$, $\lambda = \frac{1}{2(N+M)}$, and the result follows from~\eq{eq:lin}, setting $R = \frac{NM}{(N+M)^2} \big( (q_1 + q_2) - (p_1 + p_2) \big) \E \big[ X - Y | NX + MY \big]$ and Lemmas~\ref{lem:d2mom} and~\ref{lem:d3mom}.

Recalling \eq{eq:XX} and \eq{eq:YY}, then by taking expectations of both equations, we can solve for $\E[X]$ and $\E[Y]$. Similarly for the second order moments, recall~\eq{eq:x2mom}, a symmetric equation for the second moment of $Y$ and the following,
\ba{
NM \E[X'Y' | X,Y] &= \E[ (A + C)(B+D)|X,Y]\\
	&= [ (N-c) (p_1 + (1-p_1-p_2)X) + c(q_1 + (1-q_1-q_2)Y) ]\\
	&\ \ \ \cdot [ (M-c)(q_1 + (1-q_1-q_2)Y) + c(p_1 + (1-p_1-p_2)X)].
}
Again taking expectations, we can solve the above system of equations for all the required moments, and the leading order term can be found in an elementary manner. We present only the leading order term as the complete algebraic expressions for the moments are far too long to be included in this paper. 
\end{proof}
\subsection{Proof of theorem~\ref{thm:SBB}}
As before our goal is to utilise Theorem~3 from~\cite{GRR17} to bound the difference between $\frac{N}{N+M}X + \frac{M}{N+M}Y$ and a beta random variable. To show the linearity condition required to use~\cite[Theorem~3]{GRR17}, we first note that $NX' + MY' | X,Y \ed A|(X,Y) + MY$ where $A|(X,Y) \sim \Bin(N, p_1 + (1-p_1-p_2)X)$. Then
\ba{
\E[ NX' + MY' | X,Y] = N(p_1 + (1-p_1 - p_2)X) + MY,
}
and hence using the tower property
\ban{
\frac{1}{N+M} &\E[ NX' + MY' - (NX + MY) | NX + MY]\notag\\
	&=\frac{1}{N+M} \E[ Np_1 - N(p_1 + p_2)X | NX + MY]\notag\\
	&= \frac{N}{2(N+M)^2} \Big[ 2(N+M)p_1 - 2(N+M)(p_1 + p_2) \Big(\frac{N}{N+M}X + \frac{M}{N+M}Y\Big) \Big]\notag\\
	& \ \ \ - \frac{NM}{(N+M)^2}(p_1 + p_2) \E[X-Y | NX + MY].\label{eq:lin2}
}
We set $\lambda = \frac{N}{2(N+M)^2}$ and bound the necessary terms in the following lemmas.
\begin{lemma}\label{lem:s2mom}
\ba{
\E \Bigg| \Big[ &\Big(\frac{N}{N+M} X + \frac{M}{N+M}Y \Big)\Big( 1- \Big( \frac{N}{N+M} X + \frac{M}{N+M}Y\Big) \Big)\Big] \\
	&- \frac{1}{2\lambda} \E \Big[ \Big(\Big( \frac{N}{N+M} X' + \frac{M}{N+M}Y' \Big) - \Big( \frac{N}{N+M} X + \frac{M}{N+M}Y \Big)\Big)^2 \Big| NX + MY \Big] \Bigg|\\
	&\leq N(p_1 + p_2)^2 + (p_1 + p_2) + \frac{M}{N+M}\sqrt{\E (X - Y)^2}.
}
\end{lemma}
\begin{proof}
\ba{
\E\big[ &\big((NX' + MY') - (NX + MY)\big)^2 | X,Y \big]  = \E \big[ (A - NX)^2 | X,Y \big]\\
	&= \E \big[ A(A-1) + A - 2ANX + N^2X^2 | X,Y\big]\\
	&= N(N-1)(p_1 + (1-p_1-p_2)X)^2 + (1-2NX)N(p_1 + (1-p_1+p_2)X) + N^2 X^2\\
	&= X^2 \big[ N(N-1) (1-p_1-p_2)^2 - 2N^2(1-p_1-p_2) + N^2 \big]\\
	&\ \ \ + X \big[ 2N(N-1)p_1(1-p_1-p_2) -2N^2 p_1 + N(1-p_1-p_2) \big] + \big[N(N-1)p_1^2 + Np_1\big]\\
	&= X^2\big[ -N + N^2(p_1 + p_2)^2 + N( 2(p_1 + p_2) - (p_1 + p_2)^2) \big]\\
	&\ \ \ + X \big[ N -2N^2 p_1(p_1 + p_2) + N(-2p_1(1-p_1-p_2) - (p_1 + p_2))\big]\\
	&\ \ \ + \big[ N(N-1)p_1^2 + Np_1\big].
}
Therefore
\ba{
&\frac{1}{2\lambda(N+M)^2} \E\big[ \big((NX' + MY') - (NX + MY)\big)^2 | NX + MY]\\
	&= \Big(\frac{N}{N+M} X + \frac{M}{N+M}Y \Big)\Big( 1- \Big( \frac{N}{N+M} X + \frac{M}{N+M}Y\Big) \Big)\\
	&\ \ \ +  \E \Bigg[ - \frac{(2NM + M^2)}{(N+M)^2}X^2 + \frac{M^2}{(N+M)^2} Y^2 + \frac{2NM}{(N+M)^2}XY + \frac{M}{N+M}(X-Y)\Bigg| NX + MY \Bigg]\\
	&\ \ \ +\frac{1}{N}\E \Big[X^2\big[ N^2(p_1 + p_2)^2 + N( 2(p_1 + p_2) - (p_1 + p_2)^2) \big]\\
	&\ \ \ \ \ \  + X \big[ -2N^2 p_1(p_1 + p_2) + N(-2p_1(1-p_1-p_2) - (p_1 + p_2))\big]\\
	&\ \ \ \ \ \  + \big[ N(N-1)p_1^2 + Np_1\big] \Big| NX + MY \Big] .
}
To complete the bound we make the following observations.
\ba{
\Bigg| - \frac{(2NM + M^2)}{(N+M)^2}X^2 &+ \frac{M^2}{(N+M)^2} Y^2 + \frac{2NM}{(N+M)^2}XY + \frac{M}{N+M}(X-Y)\Bigg|\\
	&= \Bigg| \frac{M(X-Y)( N(2X-1) + M(X + Y-1))}{(N+M)^2}\Bigg|\\
	&\leq \frac{M(N+M)}{(N+M)^2} |X - Y| = \frac{M}{N+M} | X-Y|,
}
and
\ba{
\Big|X^2&\big[ N^2(p_1 + p_2)^2 + N( 2(p_1 + p_2) - (p_1 + p_2)^2) \big]\\
	&\ \ \ + X \big[ -2N^2 p_1(p_1 + p_2) + N(-2p_1(1-p_1-p_2) - (p_1 + p_2))\big] +  \big[ N(N-1)p_1^2 + Np_1\Big|\\
	&=\Big| N^2 \big[ X^2(p_1 + p_2)^2 - 2X p_1 (p_1 + p_2) + p_1^2 \big]\\
	&\ \ \ + N \big[ (2(p_1 + p_2) - (p_1 + p_2)^2) X^2 -(2p_1(1-p_1-p_2) + (p_1 + p_2))X  - p_1^2 + p_1 \big]\Big|\\
	&=\Big| N^2( X(p_1 + p_2) - p_1)^2 - N( X(p_1 + p_2) - p_1)^2  + N[ p_1(2X^2 - 3X + 1) + p_2(2X^2 -X)]\Big|\\
	&\leq N^2(p_1 + p_2)^2 + N(p_1 + p_2).
}
\end{proof}
\begin{lemma}\label{lem:s3mom}
\ba{
\frac1\lambda\E \Bigg| \Big[& \Big(\frac{N}{N+M} X' + \frac{M}{N+M}Y' \Big) - \Big(\frac{N}{N+M} X + \frac{M}{N+M}Y \Big)\Big]^3 \Bigg|\\
	& \leq \frac{2}{N(N+M)}\Big[(2N)^\frac12 + N(p_1 + p_2)\Big]^2 \Big[N^\frac12 + N(p_1 + p_2)\Big].
}
\end{lemma}
\begin{proof}
Using H\"older's inequality,
\ban{
\E \big| [(NX' &+ MY') - (NX + MY) ]^3 \big| \notag\\
	&\leq \big(\E[(NX' + MY') - (NX + MY)]^4 \big)^\frac12 \big(\E[(NX' + MY') - (NX + MY)]^2 \big)^\frac12. \label{eq:sbd1}
}
Recall $NX' + MY' | (X,Y) \ed A + MY$ where $A|(X,Y) \sim \Bin(N, p_1 + (1-p_1-p_2)X)$. Then
\ban{
\E\big(& (NX' + MY') - (NX + MY)\big)^4\notag \\
	&=\E\big[ \E[ ((NX' + MY') - (NX + MY))^4|X,Y]\big]\notag \\
	&= \E \big[ \E[ (A-NX)^4 | X,Y] \big]\notag \\
	&\leq \E\Big[\big(\E( (A - \E(A|X,Y))^4 | X,Y)\big)^{\frac14} + [ \E( \E(A|(X,Y) - NX)^4)]^{\frac14} \Big]^{4}\notag\\
	%	&\leq (4N^2)^\frac14 + \E [N(p_1 - (p_1 + p_2)X)]\\
	&\leq \big[ (4N^2)^{\frac14} + N(p_1 + p_2)\big]^4.\label{eq:sbd2}
}
where we used Minkowski's inequality in the first inequality and~\eq{binom4} and $|\E(A | XY) - NX| \leq N(p_1 + p_2)$ in the final inequality. Similarly again using Minkowski's inequality and~\eq{binom2},
\ban{
\E \big( &(NX' + MY') - (NX + MY)\big)^2\notag\\
	&= \E\big[\E[(A - NX)^2 | X,Y]\big]\notag\\
	&\leq \E\Big[\big[ \E[(A - \E(A|X,Y))^2|X,Y ]^\frac12 + \E[ \E ( (A|X,Y) - NX)^2]^\frac12 \big]^2\Big]\notag\\
	&\leq \big[ N^\frac12 + N(p_1 + p_2)\big]^2.\label{eq:sbd3}
}
The final bound now follows from using~\eq{eq:sbd2} and~\eq{eq:sbd3} in~\eq{eq:sbd1}. 
\end{proof}
\begin{proof}[Proof of Theorem~\ref{thm:SBB}]
We apply~\cite[Theorem 3]{GRR17}, set $a_1 = 2(N+M)p_1, a_2 = 2(N+M)p_2$, $\lambda = \frac{N}{2(N+M)^2}$, and the result follows from~\eq{eq:lin2}, setting $R = -\frac{NM}{(N+M)^2} (p_1 + p_2) \E \big[ X - Y | NX + MY \big]$ and Lemmas~\ref{lem:s2mom} and~\ref{lem:s3mom}.

For the final moment bound, similar to the proof of Theorem~\ref{thm:Beta}, recalling~\eq{eq:XX2},~\eq{eq:YY2}, \eq{sbmomx2}, \eq{sbmomy2}, adapting~\eq{sbmomxy} and taking expectations, it is elementary to solve for $\E(X-Y)^2$. As for the two-island model, the complete algebraic expressions are extremely long and complicated, so we present only the leading order term.
\end{proof}

\subsection{Proof of Lemma~\ref{lem:TIbeta}}
\begin{proof}[Proof of Lemma~\ref{lem:TIbeta}]
The following equations can be found by taking expectations of~\eq{eq:gen} with respect to $f(x,y) = x, f(x,y) = y, f(x,y) = x^2, f(x,y) = y^2, f(x,y) = xy$, or~\cite[Lemma~2.7]{Blathetal2019}.
\ba{
0&=\E \left[a_1 - (a_1 + a_2 )X + c(Y - X) \right] ,\\
0&=\E \left[b_1 - (b_1 + b_2)Y + \gamma c(X - Y)\right],\\
0&=\E \left[2X (a_1 - (a_1 + a_2)X + c(Y-X)) + \alpha X(1-X) \right],\\
0&=\E \left[2Y(b_1 - (b_1 + b_2)Y + \gamma c(X-Y)) + \beta Y(1-Y) \right],\\
0&=\E \left[ Y(a_1 - (a_1 + a_2)X + c(Y-X)) + X(b_1 - (b_1 + b_2)Y + \gamma c(X-Y)) \right].
}
This system of equations can be solved by elementary means and as $c \to \infty$,
\ba{
\lim_{c \to \infty}\E(X) = \lim_{c \to \infty}\E(Y) &= \frac{\gamma a_1 + b_1}{\gamma(a_1 + a_2 )+ b_1 + b_2},
}
and
\ba{
\lim_{c \to \infty}\Var(X) &= \lim_{c \to \infty}\Var(Y) = \lim_{c \to \infty}\Cov(X,Y) \\
&= \frac{ (\gamma a_1 + b_1) [2(1+\gamma)(\gamma a_1 + b_1 + \gamma^2 \alpha + \beta)]}{(\gamma(a_1 + a_2) + b_1 + b_2) [ 2(1 + \gamma) ( \gamma(a_1 + a_2) + (b_1 + b_2)) + \gamma^2\alpha + \beta]}.
}
Hence as $\lim_{c \to \infty}\E(X - Y) = \lim_{c \to \infty}\Var(X - Y) = 0$, $X - Y \stackrel{p}{\to} 0$ as $c \to \infty$.
\end{proof}
\bibliographystyle{apalike}
\bibliography{references}

\end{document}